\documentclass[12pt,reqno]{amsart}
\usepackage{amsmath}
\usepackage{amssymb}
\usepackage{amstext}
\usepackage{a4wide}
\usepackage{graphicx}
\usepackage{lineno}
\usepackage[numbers,sort&compress]{natbib}
\allowdisplaybreaks \numberwithin{equation}{section}
\usepackage{color}
\usepackage{cases}

\numberwithin{equation}{section}

\newtheorem{theorem}{Theorem}[section]
\newtheorem{proposition}[theorem]{Proposition}

\newtheorem{lemma}[theorem]{Lemma}
\newtheorem*{Yudovich's Theorem}{Yudovich's Theorem}

\theoremstyle{definition}

\newtheorem{definition}[theorem]{Definition}

\theoremstyle{remark}
\newtheorem{remark}[theorem]{Remark}

\begin{document}

\title
[On concentrated traveling vortex pairs]{On concentrated traveling vortex pairs with prescribed impulse}

 \author{Guodong Wang}

\address{Institute for Advanced Study in Mathematics, Harbin Institute of Technology, Harbin 150001, P.R. China}
\email{wangguodong@hit.edu.cn}


\begin{abstract}
In this paper, we consider a constrained maximization problem related to planar vortex pairs with prescribed impulse. 
We prove existence, stability and  asymptotic behavior for the maximizers, hence obtain a family of stable traveling vortex pairs  approaching a pair of point vortices with equal magnitude and opposite signs. As a corollary, we get fine asymptotic estimates for Burton's vortex pairs with large impulse. For the non-concentrated case,  
we prove a form of stability for the Chaplygin-Lamb dipole.
 \end{abstract}

\maketitle

\section{Introduction and main result}

For a two-dimensional ideal fluid of unit density with vanishing velocity at infinity, the governing  equations take the form
\begin{equation}\label{vor0}
\begin{cases}
\partial_t\omega+\mathbf v\cdot\nabla\omega=0, &  \mathbf x=(x_1,x_2)\in\mathbb R^2,\,\,t>0,\\
\mathbf v(t,\mathbf x)=-\frac{1}{2\pi}\int_{\mathbb R^2}\frac{(\mathbf x-\mathbf y)^\perp}{|\mathbf x-\mathbf y|^2}\omega(t,\mathbf y)d\mathbf y,
\end{cases}
\end{equation}
where  $\mathbf v=(v_1,v_2)$ is the velocity field,  $\omega=\partial_{x_1}v_2-\partial_{x_2}v_1$ is the scalar vorticity, and $\perp $ denotes the clockwise rotation through $\pi/2$, that is, $ (x_1,x_2)^\perp=( x_2,-x_1)$. 
For initial vorticity $\omega|_{t=0}=\omega_0\in L^1\cap L^\infty(\mathbb R^2),$ global existence and uniqueness of weak solutions for the vorticity equation \eqref{vor0}  was established by Yudovich \cite{Y}. A modern proof of Yudovich's result can   be found in \cite{MB} or \cite{MP}.

 In this paper, we are interested in a special class of solutions of Yudovich type, called \emph{traveling vortex pairs}, that is,  weak solutions exhibiting odd symmetry with respect to some line   and traveling along the direction of the line at a constant speed without changing the profile. After a suitable translation and rotation, we can assume that $L$ coincides with the $x_1$-axis, then a vortex pair  $\omega$ traveling with speed $b$ takes the form
\begin{equation}\label{tvp}
\omega(t,\mathbf x)=\zeta(x_1-bt,x_2)-\zeta(x_1-bt,-x_2) \quad\forall\,\mathbf x=(x_1,x_2)\in\mathbb R^2,
\end{equation}
where $\zeta$ is an integrable and bounded function supported in the upper half-plane $\Pi=\{\mathbf x=(x_1,x_2)\in\mathbb R^2\mid x_2>0\}.$ In the literature, it is often the case that $\zeta$ has bounded support.

When  studying  vortex pairs, it is  convenient to consider the following vorticity  equation in the upper half-plane
\begin{equation}\label{vor1}
\begin{cases}
\partial_t\omega+\nabla^\perp\mathcal G\omega\cdot\nabla\omega=0, &  \mathbf x=(x_1,x_2)\in \Pi,\,\,t>0,\\
 \mathcal G\omega(t,\mathbf x)=\frac{1}{2\pi}\int_\Pi\ln\frac{|\mathbf x-\bar{\mathbf y}|}{|\mathbf x-\mathbf y|}\omega(t,\mathbf y)d\mathbf y.
\end{cases}
\end{equation}
Here $\nabla^\perp=(\partial_{x_2},-\partial_{x_1}),$ and  $\bar{\mathbf y }$  denotes the reflection  of $\mathbf y$ in the $x_1$-axis, i.e., $\bar{\mathbf y}=(y_1,-y_2).$
Once we have obtained a traveling solution  to \eqref{vor1}, we can extend it to the whole plane by reflection in the $x_1$-axis to get a traveling vortex pair for \eqref{vor0}. The vorticity equation \eqref{vor1} can also be regarded as the governing equation for  an ideal  fluid of unit density such that the velocity $\mathbf v=(v_1,v_2)$ satisfies  
\[v_2\equiv0\mbox{ on }\partial \Pi=\{\mathbf x\in\mathbb R^2\mid x_2=0\}\,\,\mbox{ and }\,\, \lim_{|\mathbf x|\to+\infty}|\mathbf v(t,\mathbf x)|=0,\quad\forall\,t\geq0.\]

The study for traveling vortex pairs has lasted for more than one century.
For experimental or numerical results, see \cite{DS,OZ,Pu} for example. For rigorous existence results, see \cite{B11,B6,B10,CLZ,HM,Lamb,Nor,SV,T,Yang1,Yang2} and the references therein.  One famous traveling vortex pair is the \emph{Chaplygin-Lamb dipole}, independently introduced by Chaplygin \cite{Chap} and Lamb \cite{Lamb} in the early 20th century. For the Chaplygin-Lamb dipole, the stream function  can be explicitly expressed  in terms of the Bessel function of the first kind (\cite{Lamb}, p.245), and the support of the vorticity is exactly a disk (therefore the Chaplygin-Lamb dipole is also called Lamb's circular vortex pair).  A form of orbital stability for the Chaplygin-Lamb dipole was recently proved by Abe-Choi \cite{AC}. In Section 5, we will consider the  Chaplygin-Lamb dipole and prove a different kind of stability.
Another typical example is a pair of \emph{point vortices} with equal magnitude and opposite signs, taking the form
\begin{equation}\label{frm}
\omega=\kappa \delta_{\mathbf z(t)}-\kappa\delta_{\bar{\mathbf z}(t)}, \quad\mathbf z(t)=(bt,d),
\end{equation}
where $b$ represents the traveling speed, $d>0$ is fixed,  and $\delta_{\mathbf z(t)}$ is the Dirac measure centered on $\mathbf z(t).$   Note that \eqref{frm} is just a formal singular solution to the vorticity equation \eqref{vor0}, not in the sense of Yudovich. For a rigorous discussion on the relation between point vortices and the vorticity equation \eqref{vor0}, we refer  the interested reader to Marchioro-Pulvirenti \cite{MP0}. According to \cite{MP0}, $\mathbf z(t)$ in \eqref{frm} is supposed to satisfy the following Kirchhoff-Routh equation
\[\frac{d\mathbf z(t)}{dt}=\frac{\kappa}{2\pi}\frac{(\mathbf z(t)-\bar{\mathbf z}(t))^\perp}{|\mathbf z(t)-\bar{\mathbf z}(t)|^2}\]
Therefore $\kappa,b,d$ necessarily satisfy the formula
\begin{equation}\label{relation}
 4\pi bd=\kappa.
\end{equation}

An important and interesting problem related to \eqref{frm} is the \emph{desingularization problem}, that is,
to find a family of regular solutions in the sense of Yudovich such that the vorticity is concentrated and localized near the two points $(0,d)$ and $(0,-d)$ and approximate the singular solution \eqref{frm}.   More precisely, one needs to construct a family of concentrated vortex pairs $\{\omega^\varepsilon\},$ where $\varepsilon>0$ is a parameter,  such that
\[\omega^\varepsilon(t,\mathbf x)=\zeta^\varepsilon(x_1-bt,x_2)-\zeta^\varepsilon(x_1-bt,-x_2),\quad \int_\Pi\zeta^\varepsilon(\mathbf x)d\mathbf x=\kappa,\quad{\rm supp}(\zeta^\varepsilon)\subset B_{o(1)}(0,d), \]
where   $o(1)\to 0$ as $\varepsilon\to0^+.$
This problem can be studied via several methods, including the vorticity method (see \cite{CLZ,T}),  the stream function method (see \cite{SV,Yang2}), and the contour dynamics reformulation (see \cite{HM}). Now that  the Chaplygin-Lamb dipole is not a vortex pair of the above desingularization type since its support  touches the $x_1$-axis.

Although there are  various existence results  on  concentrated traveling vortex pairs in the literature, the corresponding stability problem is less well studied. In this paper,  by stability we mean  Lyapunov type. More  specifically,   a vortex pair $\bar\omega$ is said to be stable if for any initial vorticity $\omega_0$  that is close to $\bar\omega$, the evolved vorticity $\omega_t$ remains close to $\bar\omega$ for all $t>0$. Since a vortex pair translates along the $x_1$ direction at a constant speed, it is more reasonable to consider orbital stability, the precise meaning of which  will be given in Theorem \ref{thm1}. It is worth mentioning that in most cases  stability is a more sophisticated problem than existence, since it usually requires better estimates.  To our limited knowledge,  there are very few  results on the stability of concentrated traveling vortex pairs in the literature.

Our main purpose in this paper is to prove the existence of a large class of   concentrated traveling vortex pairs with certain stability.  To make our main result concise, we introduce the concept of \emph{$L^p$-regular solutions}. Throughout this paper, let $2<p<+\infty$ be fixed.

\begin{definition}\label{def1}
An $L^p$-regular solution is a function $\omega\in C([0,+\infty);L^1 \cap  L^p(\Pi))$ such that 
\begin{itemize}
\item[(i)] $\omega$ solves the vorticity equation \eqref{vor1} in the distributional sense, that is,
\begin{equation*}
  \int_\Pi\omega(0,\mathbf x)\varphi(0,\mathbf x)d\mathbf x+\int_0^{+\infty}\int_\Pi\omega\left(\partial_t\varphi+\nabla^\perp \mathcal G\omega\cdot \nabla\varphi \right) d\mathbf xdt=0\quad\forall\,\varphi\in C_c^{\infty}(\Pi\times\mathbb R);
  \end{equation*}

\item[(ii)]  the kinetic energy $E$ and the impulse (parallel to the $x_1$-axis) $I$, defined by  
\begin{equation}\label{dev1}
E(\omega(t,\cdot))=\frac{1}{2}\int_\Pi \omega(t,\mathbf x)\mathcal G\omega(t,\mathbf x)d\mathbf x,\quad I(\omega(t,\cdot))=\int_\Pi x_2\omega(t,\mathbf x)d\mathbf x,
\end{equation}
are conserved, that is, for all $t>0$,  
\[E(\omega(t,\cdot))=E(\omega(0,\cdot)), \quad I(\omega(t,\cdot))=I(\omega(0,\cdot)).\]
\end{itemize} 
\end{definition}

Note that for any initial vorticity  $\omega_0\in L^1\cap L^p(\Pi)$ with bounded support, an $L^p$-regular solution always exists, but may not be unique, except in the case of $\omega_0$ being additionally  bounded. 
See \cite{B6}  for a detailed discussion on this issue.

We also need some notation. In the rest of this paper,
let $ \varrho\in L^p(\mathbb R^2)$ be given such that  \begin{itemize}
\item[(H1)] $\varrho$ is nontrivial and nonnegative;
\item[(H2)] $\varrho$ has compact support;
\item[(H3)] $\varrho$   is radially symmetric and decreasing, that is, 
\[\varrho(\mathbf x)=\varrho(\mathbf y)\quad \forall\,\mathbf x,\mathbf y\in\mathbb R^2 \mbox{ such that } |\mathbf x|=|\mathbf y|,\]
\[ \varrho(\mathbf x)\geq \varrho(\mathbf y)\quad \forall\,\mathbf x,\mathbf y\in\mathbb R^2 \mbox{ such that } |\mathbf x|\leq |\mathbf y|.\]
\end{itemize}
By (H1)-(H3), there exists some $r>0$  such that
\begin{equation}\label{dero}
\{\mathbf x\in\mathbb R^2\mid \varrho(\mathbf x)>0\} =B_r(\mathbf 0).
\end{equation}
Here  $\mathbf 0$ denotes the origin. Let $\kappa>0$ be fixed such that
\begin{equation}\label{deropq}
\kappa=\int_{\mathbb R^2}\varrho(\mathbf x)d\mathbf x.
\end{equation}
For $\varepsilon>0$, denote
\begin{equation}\label{scad}
\varrho^\varepsilon(\mathbf x)=\frac{1}{\varepsilon^2}\varrho\left(\frac{\mathbf x}{\varepsilon}\right),\quad \mathbf x\in\mathbb R^2.
\end{equation}
Let $\mathcal R(\varrho^\varepsilon)$ be  the set of all equimeasurable rearrangements of $\varrho^\varepsilon$ on $\Pi$, that is,
\[\mathcal R(\varrho^\varepsilon)=\{v\in L^p(\Pi)\mid \mathcal L\left(\{\mathbf x\in \Pi\mid v(\mathbf x)>s\}\right) =\mathcal L\left(\{\mathbf x\in \mathbb R^2\mid \varrho^\varepsilon(\mathbf x)>s\}\right)\,\,\forall\,s\in\mathbb R\}.\]
Here and henceforth, $\mathcal L$ denotes the two-dimensional Lebesgue measure.

For fixed $i_0>0$, define
\begin{equation}\label{sdef}
\mathcal S_{\varepsilon,i_0}=\{v\in\mathcal R(\varrho^\varepsilon)\mid I(v)=i_0\}.
\end{equation}
Consider the following maximization problem
\begin{equation}\label{maxim}
M_{\varepsilon,i_0}=\sup_{v\in\mathcal S_{\varepsilon,i_0}}E(v).
\end{equation}
Denote $\Gamma_{\varepsilon,i_0}$ the set of maximizers of $E$ over $\mathcal S_{\varepsilon,i_0}$, that is,
\begin{equation}\label{maximu}
\Gamma_{\varepsilon,i_0}=\left\{v\in\mathcal S_{\varepsilon,i_0}\mid E(v)=M_{\varepsilon,i_0}\right\}.
\end{equation}
 Note that $\Gamma_{\varepsilon,i_0}$ might be   empty.

Our main result  can be summarized as follows.
\begin{theorem}\label{thm1}
Let $p>2$ and $i_0>0$ be given. Let $\varrho\in L^p(\Pi)$ satisfy (H1)-(H3). Let $\Gamma_{\varepsilon,i_0}$ be defined by \eqref{maximu}.
Then there exists some $\varepsilon_0>0$, depending only on $\varrho$ and $i_0$, such that for any $\varepsilon\in(0,\varepsilon_0)$, the following  conclusions hold.
\begin{itemize}
\item [(1)] $\Gamma_{\varepsilon,i_0}\neq \varnothing.$
\item [(2)] Every $\zeta\in\Gamma_{\varepsilon,i_0}$ has bounded support, and is Steiner-symmetric with respect to the line $x_1=c$ for some $c\in\mathbb R.$
\item [(3)] For every $\zeta\in\Gamma_{\varepsilon,i_0}$, there exists some positive number $\lambda_\zeta$ and some increasing function $f_\zeta:\mathbb R\to\mathbb R\cup\{\pm\infty\}$  such that 
\begin{equation}
\zeta=f_\zeta(\mathcal G\zeta-\lambda_\zeta x_2) \,\,\mbox{ \rm a.e. in }   \Pi,
\end{equation}
\begin{equation}
\zeta=0 \,\,\mbox{ \rm a.e. in }  \{\mathbf x\in\Pi\mid \mathcal G\zeta(\mathbf x)-\lambda_\zeta x_2\leq 0\}.
\end{equation}
\item[(4)] For every $\zeta\in\Gamma_{\varepsilon,i_0},$ $\omega(t,\mathbf x)=\zeta(x_1-\lambda_\zeta t,x_2)-\zeta(x_1-\lambda_\zeta t,-x_2)$   is an $L^p$-regular solution to the vorticity \eqref{vor1},  thus yielding a traveling vortex pair with speed $\lambda_\zeta$.
\item[(5)]
Define
\begin{equation}\label{dev2}
\|v\|_{\mathfrak X_p}=|I(v)|+\|v\|_{L^1(\Pi)}+\|v\|_{L^p(\Pi)}.
\end{equation}
Then  $\Gamma_{\varepsilon,i_0}$ is orbitally  stable in the $\mathfrak X_p$ norm, that is, for any $\epsilon>0,$ there exists some $\delta>0$, such that
 for any  $L^p$-regular solution $\omega$ satisfying  
 \[\omega(0,\cdot) \mbox{ \rm is nonnegative with  bounded support,}\quad \inf_{\zeta\in\Gamma_{\varepsilon,i_0}}\|\omega(0,\cdot)-\zeta\|_{\mathfrak X_p}<\delta,\]
 it holds that
\[\inf_{\zeta\in\Gamma_{\varepsilon,i_0}}\|\omega(t,\cdot)-\zeta\|_{\mathfrak X_p}<\epsilon\quad\forall \,t\ge 0.\]
 
\item[(6)] For every $\zeta\in\Gamma_{\varepsilon,i_0},$ denote  $V_\zeta=\{\mathbf x\in\Pi\mid \zeta(\mathbf x)>0\}$. Then there is some $C>0$,  depending only on  $\varrho$ and $ i_0,$ such that
\[{\rm diam}(V_\zeta)\leq C\varepsilon\quad\forall\,\zeta\in \Gamma_{\varepsilon,i_0},\]
where ${\rm diam}(V_\zeta)$ is the diameter of $V_\zeta$ given by
\[{\rm diam}(V_\zeta)=\sup_{\mathbf x,\mathbf y\in V_\zeta}|\mathbf x-\mathbf y|.\]


\item[(7)] As $\varepsilon\to0^+$, $\lambda_\zeta\to  {\kappa^2}/(4\pi i_0)$ uniformly for $\zeta\in \Gamma_{\varepsilon,i_0}$. More precisely, for any $\epsilon>0,$ there exists some  $\varepsilon_2\in(0,\varepsilon_0),$ depending only on $\varrho,i_0$ and $\epsilon,$  such that for any $\varepsilon\in(0,\varepsilon_2),$ it holds that
\[\left|\lambda_\zeta- \frac{\kappa^2}{4\pi i_0}\right|<\epsilon\quad\forall\,\zeta\in\Gamma_{\varepsilon,i_0}.\]

\item[(8)] Define
\begin{equation}\label{gam0}
\Gamma^0_{\varepsilon,i_0}=\left\{\zeta\in \Gamma_{\varepsilon,i_0}\,\,\bigg|\,\, \int_\Pi x_1\zeta(\mathbf x)d\mathbf x=0\right\}.
\end{equation}
For $\zeta\in \Gamma^0_{\varepsilon,i_0}$, extend $\zeta$ to $\mathbb R^2$ such that $\zeta\equiv 0$ in the lower half-plane. Define 
\[\nu^{\zeta,\varepsilon}(\mathbf x)=\varepsilon^2\zeta(\varepsilon\mathbf x+\hat{\mathbf x}),\quad\hat{\mathbf x}=(0,i_0/\kappa).\]
 Then  $\nu^ {\zeta,\varepsilon}\to\varrho$ in $L^p(\mathbb R^2)$ uniformly as $\varepsilon\to0^+$.
More precisely, for any $\epsilon>0,$ there exists some  $\varepsilon_3\in(0,\varepsilon_0),$ depending only on $\varrho,i_0,p$ and $\epsilon,$  such that for any $\varepsilon\in(0,\varepsilon_3),$ it holds that
\[\|\nu^{\zeta,\epsilon}-\varrho\|_{L^p(\mathbb R^2)}<\epsilon\quad\forall\,\zeta\in\Gamma^0_{\varepsilon,i_0}.\]
\end{itemize}
\end{theorem}

\begin{remark}\label{rmk1}
Since  $E$ and $I$ are both invariant under any translation parallel to the $x_1$-axis, it is clear that
\begin{equation}\label{ccto}
\Gamma_{\varepsilon,i_0}=\{\zeta(\cdot+c\mathbf e_1)\mid \zeta\in\Gamma^0_{\varepsilon,i_0},\,\,c\in\mathbb R\},
\end{equation}
where $\mathbf e_1=(1,0).$
Therefore $\Gamma_{\varepsilon,i_0}^0$ is exactly the set of maximizers that are Steiner-symmetric in the $x_2$-axis.
As a result, by \eqref{ccto} and item (8) in Theorem \ref{thm1}, we see that after some suitable translation and scaling,  any $\zeta\in \Gamma_{\varepsilon,i_0}$ converges uniformly to $\varrho$ in $L^p(\mathbb R^2)$.
\end{remark}

\begin{remark}\label{rmk3}
By the definition of  $ \Gamma^0_{\varepsilon,i_0}$ and the fact that $I(\zeta)=i_0$ for any $\zeta\in \Gamma_{\varepsilon,i_0},$ we see that
\[\frac{1}{\kappa}\int_{\Pi}\mathbf x\zeta(\mathbf x)d\mathbf x=\left(0, \frac{i_0}{\kappa}\right)=\hat{\mathbf x}.\]
That is, for any $ \zeta\in \Gamma^0_{\varepsilon,i_0}$ the vorticity center is exactly $\hat{\mathbf x}.$ Taking into account item (6) in Theorem \ref{thm1}, we  deduce that
\[V_\zeta\subset B_{C\varepsilon}(\hat{\mathbf x})\quad\forall\,\zeta\in\Gamma^0_{\varepsilon,i_0}.\]
Therefore Theorem \ref{thm1} in fact provides a family of desingularized solutions for the following pair of point vortices
  \begin{equation}\label{frm2}
\omega=\kappa \delta_{\mathbf z(t)}-\kappa\delta_{\bar{\mathbf z}(t)}, \quad\mathbf z(t)=\left(\frac{\kappa^2 t}{4\pi i_0},\frac{i_0}{\kappa}\right).
\end{equation}
\end{remark}

The proofs of items (1)-(5) in Theorem \ref{thm1} follow from Burton's results in \cite{B10} and the scaling properties of $\mathcal S_{\varepsilon,i_0}$ and $\Gamma_{\varepsilon,i_0}$. Our focus in this paper is the asymptotic behavior (6)-(8).   To prove (6)-(8), we adapt the energy method established by Turkington in \cite{T}. However, unlike the bounded domain case in \cite{T}, in this paper  the unboundedness of the upper half-plane causes great trouble since the regular part of the Green function is no longer bounded from below. 
To overcome this difficulty, we need to show the uniform boundedness of the vortex core. This is achieved based on the idea developed by the author  in \cite{W2}. 
Of course, due to the different variational nature,  some new technical difficulties appear during this process, most of which are caused by the traveling speed.  In contrast to \cite{W2},  the vortex pairs in this paper have unknown traveling speed that depends on the  parameter $\varepsilon$.  To apply the method in \cite{W2}, some proper uniform estimates for the traveling speed should be deduced at the  beginning. To achieve this, we use a Pohozaev type identity proved by Burton in \cite{B11} and the fact that the  kinetic energy on the vortex core is uniformly bounded.  As to the limit of the traveling speed as $\varepsilon\to 0^+$, we derive a useful formula (see Lemma \ref{ffs1}) such that the traveling speed can be expressed in terms of an integral involving the vorticity.  Except for the traveling speed, some specific estimates are also different from the ones in \cite{W2} and need to be taken seriously.

Below we recall  some  closely related results in the literature and give some comments on Theorem \ref{thm1}.
In \cite{T}, p.1061, Turkington  proposed a possible method to construct  concentrated vortex  pairs with prescribed impulse, that is, to  maximize $E$ over
\[J_{a,\lambda}=\left\{v\in L^\infty(\Pi)\mid \|v\|_{L^1(\Pi)}=\kappa,\,\,I(v)=i_0,\,\,0\leq v\leq \lambda,\,\,{\rm supp}(v)\subset [-a,a]\times[0,a]\right\},\]
where $\kappa, i_0>0$ are fixed, $a>i_0/\kappa$, and $\lambda>0$ is a  parameter.    Following Turkington's idea in \cite{T}, it is not hard to prove that 
 for sufficiently large $\lambda$, depending on $\kappa,i_0$ and $a$,  there exists a maximizer, and any maximizer $\zeta $ satisfies 
\begin{equation}\label{pfrm}
\zeta =\lambda\mathbf 1_{\Omega_\zeta}
\end{equation}
for some bounded open set  $\Omega_\zeta\subset\Pi$ depending on $\zeta$. Here  $\mathbf 1 $ denotes the characteristic function. Moreover, similar asymptotic estimates as in Theorem \ref{thm1}  hold  as $\lambda\to+\infty.$
Since any maximizer has the patch form \eqref{pfrm},   maximizing $E$ relative to $J_{a,\lambda}$ is actually  equivalent  to maximizing $E$ relative to
\[K_{a,\lambda}=\left\{v=\lambda\chi_\Omega \mid  \Omega\subset \Pi,\,\,\lambda\mathcal L(\Omega)=\kappa,\,\, I(v)=i_0,\,\,{\rm supp}(v)\subset [-a,a]\times[0,a]\right\}.\]
This is nearly one special case of our maximization problem \eqref{maxim} if  $\varrho$ is chosen to be  a suitable patch function, except for an extra constraint on the supports of the admissible functions. A natural question is whether this maximization problem  results in the same solutions as \eqref{maxim} does (up to a translation in the $x_1$ direction).  To answer this question, one needs to show that the parameter $a$ has no impact on the maximization problem. Unfortunately,   according to Turkington's method in \cite{T},  the asymptotic estimates for the maximizers   inevitably    depend   on $a$, and it might even happen that the diameter of the supports the maximizers go to infinity as $a\to+\infty$. One consequence of our Theorem \ref{thm1} is that it  provides  a positive answer to this question. In fact, since all the maximizers of the maximization problem \eqref{maxim} have vanishing supports near $\bar{\mathbf x}=(0,i_0/\kappa)$, the constraint ${\rm supp}(v)\subset [-a,a]\times[0,a]$ does not make any impact if $a>i_0/\kappa$.
As a corollary, the vortex patch pairs obtained by Turkington's method are orbitally stable in the sense of item (5) in Theorem \ref{thm1}.

Another related  work is \cite{B11}, in which
Burton considered a maximization problem  related to traveling vortex pairs with prescribed rearrangement and  large impulse. More specifically, Burton studied the maximization of $E$  subject to
\[\mathcal F_{a,\alpha}=\{v\in \mathcal S_{1,\alpha}\mid {\rm supp}(v)\subset [-a,a]\times[0,a]\},\]
where  $a,\alpha $ are positive constants, and  $\mathcal S_{1,\alpha}$ is defined  in analogous to \eqref{sdef}, that is,
\begin{equation}\label{s1a}
\mathcal S_{1,\alpha}=\{v\in\mathcal R(\varrho)\mid I(v)=\alpha\}.
\end{equation}
Burton proved that for sufficiently large $\alpha$ depending only on $\varrho$, and sufficiently large $a$ depending on $\varrho$ and $\alpha,$  there exists a maximizer, and any maximizer corresponds to a traveling vortex pair with unknown speed depending on $\alpha$ and $a$. Since $\mathcal F_{a,\alpha}$ is not an invariant class of the vorticity equation \eqref{vor1},  stability analysis for the vortex pairs obtained in this way is quite complicated. To obtain stable traveling vortex pairs, Burton  \cite{B10} considered the maximization of $E$ subject to $\mathcal S_{1,\alpha}$  directly. Therein, existence and a form of orbital stability were obtained based on Douglas' work \cite{Dou} and a concentrated-compactness argument.  The  precise statements of these results will be presented in Lemma \ref{lemm1},  Section 2. An interesting problem related to  Burton's vortex pairs  in \cite{B10} is to study the asymptotic behavior as $\alpha\to+\infty$, including what these vortex pairs look like  and where they are located.  As we will see in Section 4, this problem can be easily solved as a consequence   Theorem \ref{thm1}.

In addition to the vortex pairs mentioned above, which have  prescribed impulse,  in \cite{B11} and \cite{T} the authors also studied traveling vortex pairs with prescribed   speed by studying a similar variational problem.  A stability criterion for this kind of vortex pairs was  proved in \cite{B6}. Based on the results in \cite{B6} and inspired by the method in \cite{T}, recently the author  in \cite{W2}  proved the existence of a family of concentrated stable vortex pairs with prescribed  traveling speed.
It is worth mentioning that whether these two kinds of vortex pairs (that is, one with prescribed impulse and the other with prescribed speed) are the same is still open. See \cite{B10}, p. 548 for a brief discussion.

In Theorem \ref{thm1},    existence and stability are obtained for small $\varepsilon$.  For the general case, related questions are less well understood.  However, we will show in Section 5 that for some specially chosen rearrangement class and impulse, existence and stability also hold, and the maximizers are   exactly the   Chaplygin-Lamb dipole together with its  translations parallel to the $x_1$-axis.

This paper is organized as follows. In Section 2, we provide some   preliminaries for later use. In Section 3, we give the proof of Theorem \ref{thm1}. In Section 4, we  study the  asymptotic  behavior of Burton's vortex pairs in \cite{B10} based on Theorem \ref{thm1}. In Section 5, we discuss a special non-concentrated case and prove a form of stability for the Chaplygin-Lamb dipole.

\section{Preliminaries}

In this section, we present some preliminaries that will be used in Section 3.

The first lemma summarizes what Burton proved
  in \cite{B10} regarding the maximization  of $E$ relative to $\mathcal S_{1,\alpha}$, which is defined by \eqref{s1a} in Section 1. 

\begin{lemma}\label{lemm1}
Let  $\varrho $ satisfy (H1)-(H3) in Section 1,  $\alpha>0$ be a parameter, and $\Sigma_\alpha$ be the set of maximizers of $E$  relative to
$\mathcal S_{1,\alpha}$.
 Then there exists some $\alpha_0>0,$ depending only on $\varrho,$ such that for any $\alpha>\alpha_0$ the following assertions hold.
 \begin{itemize}
 \item [(i)] $\Sigma_\alpha\neq \varnothing.$
  \item[(ii)] Every $\zeta\in \Sigma_\alpha$ has bounded support, and is Steiner-symmetric with respect to some line $x_1=c$ for some $c\in\mathbb R.$
 \item [(iii)] For every $\zeta\in \Sigma_\alpha,$ there exist some positive number $q_\zeta>0$  and some increasing function $\phi_\zeta:\mathbb R\to\mathbb R\cup\{\pm\infty\}$  such that 
 \[\zeta=\phi_\zeta(\mathcal G\zeta-q_\zeta x_2)\mbox{ \rm a.e. in }\Pi,\] 
 \[\zeta=0 \mbox{ a.e. in }\{\mathbf x\in\Pi\mid \mathcal G\zeta(\mathbf x)-q_\zeta x_2\leq 0\}.\]
 \item[(iv)] For every $\zeta\in\Sigma_{\alpha},$ $\omega(t,\mathbf x)=\zeta(x_1-q_\zeta t,x_2)-\zeta(x_1-q_\zeta t,-x_2)$ is an $L^p$-regular solution to the vorticity \eqref{vor1}.
 \item[(v)]  $\Sigma_\alpha$ is orbitally stable in the $\mathfrak X_p$ norm in the following sense: for any $\epsilon>0,$ there exists some $\delta>0$, depending only on $\varrho, \alpha, p$ and $\epsilon$, 
 such that
 for any  $L^p$ regular solution $\omega$ satisfying
 \[\omega(0,\cdot) \mbox{ \rm is nonnegative with  bounded support,}\quad \inf_{\zeta\in\Sigma_{\alpha}}\|\omega(0,\cdot)-\zeta\|_{\mathfrak X_p}<\delta,\]
 it holds that
\[\inf_{\zeta\in\Sigma_\alpha}\|\omega(t,\cdot)-\zeta\|_{\mathfrak X_p}<\epsilon\quad\forall \,t\ge 0.\]
  \end{itemize}

 \end{lemma}

 \begin{proof}
 The items (i)-(iii) follow from Lemmas 9 and 10  in \cite{B10}, (iv) follows from Section 5 in \cite{B11}, and (v) follows from Theorem 2 in \cite{B10}.
 \end{proof}

\begin{remark}
In \cite{B10}, Burton in fact proved that there exist    $A_{\alpha}, B_\alpha>0$ such that 
\[{\rm supp}(\zeta)\subset [-A_{ \alpha},A_{ \alpha}]\times[0,B_{ \alpha}]\quad \forall\,\zeta\in \Sigma_\alpha.\]
However, how $A_\alpha$ and $B_\alpha$ depend  on $\alpha$  is not  specified therein.   As we will see in Section 3, this issue is  crucial  for the proof of Theorem \ref{thm1}. 
\end{remark}

\begin{lemma}\label{lemm2}
Let $1<s<+\infty$ and $0<\theta<1.$ Then there are positive numbers $c_1,c_2,c_3$, depending only on $s$ and $\theta$, such that for any $v$ satisfying
\[v\in L^s(\Pi),\quad v(\mathbf x)\geq 0 \,\,\mbox{ \rm a.e. }\mathbf x\in\Pi, \quad  I(v)<+\infty,\]
it holds that
\[\mathcal Gv(\mathbf x)\leq x_2^{-1}(c_1\ln x_2+c_2)I(v)+c_3x_2^{-\theta}\|v\|_{L^s(\Pi)}^{1-\theta}I(v)^{\theta}\quad\,\forall \,\mathbf x=(x_1,x_2)\in\Pi,\,x_2\geq 1.\]
\end{lemma}

\begin{proof}
It follows from  Lemma 5 in \cite{B10}.
\end{proof}

\begin{lemma}\label{lemm3}
Let $2<s<+\infty.$ Then there exists a positive number $K$, depending only on $s$, 
such that for any $v$ satisfying
\[v\in L^1 \cap L^s(\Pi),\quad v(\mathbf x)\geq 0 \,\,\mbox{ \rm a.e. }\mathbf x\in\Pi, \quad  v\mbox{  \rm  is Steiner-symmetric in the $x_2$-axis},\]
it holds for any $\mathbf x=(x_1,x_2)\in\Pi$ with $|x_1|\geq 1.$ that
\[\mathcal Gv(\mathbf x)\leq K\left(I(v)+\|v\|_{L^1(\Pi)}+\|v\|_{L^s(\Pi)}\right)x_2|x_1|^{-1/(2s)}.\]
\end{lemma}

\begin{proof}
It follows from  Lemma 5 in \cite{B10}.
\end{proof}

The following lemma presents two rearrangements inequalities that will be used later. 

\begin{lemma}
Let $u,v,w$ be nonnegative Lebesgue measurable functions on $\mathbb R^2$. Then
\begin{equation}\label{rri1}
\int_{\mathbb R^2}uvdx\leq \int_{\mathbb R^2}u^*v^*dx,
\end{equation}
\begin{equation}\label{rri2}
\int_{\mathbb R^2}\int_{\mathbb R^2}u(x)v(x-y)w(y)dxdy\leq \int_{\mathbb R^2}\int_{\mathbb R^2}u^*(x)v^*(x-y)w^*(y)dxdy.
\end{equation}
Here $u^*$ denote the symmetric-decreasing rearrangement of $u$.
\end{lemma}

\begin{proof}
See \S 3.4 and \S 3.7 in Lieb--Loss's book \cite{LL}.
\end{proof}

We also need the following lemma proved by Burchard--Guo.
\begin{lemma}\label{bgu}
Let $\{u_n\}_{n=1}^{+\infty}\subset L^2(\mathbb R^2)$ such that for each $n$
\begin{equation}\label{lily}
 u_n(\mathbf x) \geq 0  \mbox{ \rm a.e. }\mathbf x\in\mathbb R^2,\quad\int_{\mathbb R^2}\mathbf x u_n(\mathbf x)d\mathbf x=\mathbf 0,\quad {\rm supp}(u_n)\subset B_{\alpha}(\mathbf 0)
 \end{equation}
for some $\alpha>0$.
If $u_n\rightharpoonup u$ and $u^*_n\rightharpoonup v$ for some $u,v\in L^2(\mathbb R^2),$ then
\[\int_{\mathbb R^2}\int_{\mathbb R^2}\ln\frac{1}{|\mathbf x-\mathbf y|}u(\mathbf x)u(\mathbf y)d\mathbf xd \mathbf y\leq \int_{\mathbb R^2}\int_{\mathbb R^2}\ln\frac{1}{|\mathbf x-\mathbf y|}v(\mathbf x)v(\mathbf y)d\mathbf xd \mathbf y.\]
Moreover, the equality holds if and only $u=v$.
\end{lemma}

\begin{proof}

It is a  special case of Lemma 3.2 in \cite{BGu}.
\end{proof}

\section{Proof of Theorem \ref{thm1}}

In this section, we give the proof of Theorem \ref{thm1}. 
For clarity, we divide the proof into three subsections.

\subsection{Existence and stability}

The aim of this subsection is to prove items (1)-(4) in Theorem \ref{thm1} based on Lemma \ref{lemm1}.

To begin with,  we prove the following scaling properties of $\mathcal S_{\varepsilon,i_0}$ and $ \Gamma_{\varepsilon,i_0}$ (see Section 1 for their definitions).

\begin{lemma}\label{lam1}
For any function $v:\Pi \to\mathbb R$, we have 
\begin{itemize}
\item [(i)]  $v\in \mathcal S_{1,i_0/\varepsilon}$ if and only if  $v^\varepsilon\in\mathcal S_{\varepsilon,i_0};$
\item[(ii)] $v\in \Gamma_{1,i_0/\varepsilon}$ if and only if  $v^\varepsilon\in\Gamma_{\varepsilon,i_0}.$
\end{itemize}
\end{lemma}
\begin{proof}
By  definition, $v\in\mathcal S_{1,i_0/\varepsilon}$ if and only if $v\in\mathcal R(\varrho)$ and $I(v)=i_0/\varepsilon.$ It is clear that $v\in\mathcal R(\varrho)$ if and only if $v^\varepsilon\in\mathcal R(\varrho^\varepsilon)$. Besides, since
\[I(v^{ \varepsilon})=\int_\Pi x_2v^{ \varepsilon}(\mathbf x)d\mathbf x=\frac{1}{\varepsilon^2}\int_\Pi x_2v\left(\frac{\mathbf x}{\varepsilon}\right)d\mathbf x= {\varepsilon  }\int_\Pi x_2v(\mathbf x)d\mathbf x= \varepsilon I(v),\]
we see that $I(v)=i_0/\varepsilon$ if and only if $I(v^\varepsilon)=i_0$.  Hence (i) is proved.

To prove (ii), we first claim that
\begin{equation}\label{xdd}
E(v^\varepsilon)=E(v)\quad\forall\,v\in \mathcal S_{1,i_0/\varepsilon}.
\end{equation}
In fact, by change of variables we have
\begin{align*}
E(v^\varepsilon)&=\frac{1}{4\pi}\int_\Pi\int_\Pi\ln\frac{|\mathbf x-\bar{\mathbf y}|}{|\mathbf x-\mathbf y|}\frac{1}{\varepsilon^2}v\left(\frac{\mathbf x}{\varepsilon}\right)\frac{1}{\varepsilon^2}v\left(\frac{\mathbf y}{\varepsilon}\right)d\mathbf xd\mathbf y\\
&=\frac{1}{4\pi}\int_\Pi\int_\Pi\ln\frac{|\mathbf x-\bar{\mathbf y}|}{|\mathbf x-\mathbf y|}v\left({\mathbf x}\right)v\left({\mathbf y}\right)d\mathbf xd\mathbf y\\
&=E(v).
\end{align*}
Now (ii) is an easy consequence of (i) and \eqref{xdd}.
\end{proof}

\begin{proposition}\label{prop1}
There exists some $\varepsilon_0>0$, depending only on $\varrho$ and $i_0$, such that for any $0<\varepsilon<\varepsilon_0,$ the assertions (1)-(5) in Theorem \ref{thm1} hold.
\end{proposition}

\begin{proof}
By choosing   $\alpha=i_0/\varepsilon$ in Lemma \ref{lemm1}, we see that  for sufficiently small $\varepsilon$ the items (i)-(v)  in Lemma \ref{lemm1} hold with $\Sigma_\alpha$ replaced by $\Gamma_{1,i_0/\varepsilon}$. Then the assertions (1)-(5) in Theorem \ref{thm1}  follows immediately from (ii) in Lemma \ref{lam1}.

\end{proof}

From now on, we always assume that $\varepsilon\in(0,\varepsilon_0)$ with $\varepsilon_0$ determined in Proposition \ref{prop1}.

By   Proposition \ref{prop1}, for any
$\zeta\in\Gamma_{\varepsilon,i_0}$ there exist  some  $\lambda_\zeta>0$ and  some increasing function $f_\zeta:\mathbb R\to\mathbb R\cup\{\pm\infty\}$ such that
\begin{equation}\label{fv1}
\zeta=f_\zeta(\mathcal G\zeta-\lambda_\zeta x_2) \quad\mbox{a.e. in $\Pi$, }
\end{equation}
\begin{equation}\label{fv2}
 \zeta=0 \quad \mbox{a.e. in  } \{\mathbf x\in\Pi\mid \mathcal G\zeta(\mathbf x)-\lambda_\zeta x_2\leq 0\}.
\end{equation}
From \eqref{fv2}, we get
\begin{equation}\label{fvan}
f_\zeta(s)=0  \quad \forall\,s\in(-\infty,0].
\end{equation}
Define 
\begin{equation}\label{fvasg}
\sigma_{\zeta}=\sup_{x\in\Pi}(\mathcal G\zeta(x)-\lambda_\zeta x_2).
\end{equation}
Then it is easy to see that $f_\zeta(s)\in \mathbb R$ if $s<\sigma_\zeta$. But it may happen that 
\begin{equation}\label{fvasg2}
\lim_{s\to\sigma_\zeta^-}f_\zeta(s)=+\infty.
\end{equation}

 Define the Lagrangian multiplier $\mu_\zeta$ related to $\zeta$ as follows
 \begin{equation}\label{lmp}
 \mu_\zeta=\inf\{s\in\mathbb R\mid f_\zeta(s)>0\}.
 \end{equation}
 Then by \eqref{fvan} it holds that
 \begin{equation}\label{fvafn}
\mu_\zeta\geq 0.
\end{equation}
Following the argument as in \cite{W2}, Lemma 3.2, it is easy to check that 
 \begin{equation}\label{lesss}
V_\zeta=\{\mathbf x\in\Pi\mid \mathcal G\zeta(\mathbf x)-\lambda_\zeta x_2>\mu_\zeta\}.
\end{equation}
Here $V_\zeta$ is the vortex core related to $\zeta$ defined in Theorem \ref{thm1}.

 \subsection{Uniform estimates}
Now we deduce some uniform estimates  independent of the parameter $\varepsilon$ for the maximizers obtained in Proposition \ref{prop1}. 
In particular, we will show that all the maximizers  are  supported in some bounded region  not depending on $\varepsilon.$ 

For convenience, below we  use $C_1,C_2,C_3\cdot\cdot\cdot$ to denote positive numbers depending   on  $\varrho$ and $i_0$, but not on $\varepsilon.$

We begin with the following basic energy estimate.

\begin{lemma}\label{lam2}
There exists some $C_1>0$ such that
\begin{equation}\label{tc1u}
E(\zeta)\geq -\frac{\kappa^2}{4\pi}\ln\varepsilon-C_1\quad \forall\,\zeta\in\Gamma_{\varepsilon,i_0}.
\end{equation}
\end{lemma}

\begin{proof}
Without loss of generality assume that
\begin{equation}\label{wll}
0<\varepsilon_0<\frac{i_0}{4r\kappa},
\end{equation}
where $r$ is the positive number given in \eqref{dero}.
Define
\[\bar v(\mathbf x)=\varrho^\varepsilon(\mathbf x-\hat{\mathbf x}),\]
where  $\hat{\mathbf x}=(0,i_0/\kappa)$ as  in Theorem \ref{thm1}.
By \eqref{wll}, it is easy to check that  $\bar v\mathbf 1_\Pi\in\mathcal R(\varrho^\varepsilon)$. 
Moreover, by a simple calculation we have
\[I(\bar v\mathbf 1_\Pi)=\int_\Pi x_2\varrho^\varepsilon(\mathbf x-\hat{\mathbf x})d\mathbf x=\int_\Pi(x_2+\hat x_2)\varrho^\varepsilon(\mathbf x)d\mathbf x=\hat x_2\int_\Pi\varrho^\varepsilon(\mathbf x)d\mathbf x=i_0.\]
Here we used the fact that $\int_\Pi x_2\varrho^\varepsilon(\mathbf x)d\mathbf x=0$ due to the radial symmetry of $\varrho^\varepsilon$. Therefore we have obtained  $\bar v\mathbf 1_\Pi\in\mathcal S_{\varepsilon,i_0}$, which implies \[E(\zeta)\geq E(\bar v\mathbf 1_\Pi)\quad\forall\zeta\in\Gamma_{\varepsilon,i_0}.\]  To finish the proof, it suffices to take into account the  following estimate for $E(\bar v\mathbf 1_\Pi)$:
\begin{align}
E(\bar v\mathbf 1_\Pi)=&\frac{1}{4\pi}\int_{\Pi}\int_{\Pi}\ln\frac{|\mathbf x-\bar{\mathbf y}|}{|\mathbf x-\mathbf y|}\varrho^\varepsilon(\mathbf x-\hat{\mathbf x})\varrho^\varepsilon(\mathbf y-\hat{\mathbf x})d\mathbf xd\mathbf y\\
=&\frac{1}{4\pi}\int_{B_{r\varepsilon }(\hat{\mathbf x})}\int_{B_{r\varepsilon }(\hat{\mathbf x})}\ln\frac{|\mathbf x-\bar{\mathbf y}|}{|\mathbf x-\mathbf y|}\varrho^\varepsilon(\mathbf x-\hat{\mathbf x})\varrho^\varepsilon(\mathbf y-\hat{\mathbf x})d\mathbf xd\mathbf y \\
\geq &\frac{1}{4\pi}\int_{B_{r\varepsilon }(\hat{\mathbf x})}\int_{B_{r\varepsilon }(\hat{\mathbf x})}\ln\frac{i_0}{2r\kappa\varepsilon}\varrho^\varepsilon(\mathbf x-\hat{\mathbf x})\varrho^\varepsilon(\mathbf y-\hat{\mathbf x})d\mathbf xd\mathbf y \label{c9}\\
=&-\frac{\kappa^2}{4\pi}\ln\varepsilon-\frac{\kappa^2}{4\pi}\ln\frac{2r}{i_0}.
\end{align}
Note that in \eqref{c9} we used the following two facts:
\[|\mathbf x-\mathbf y|\leq 2r\varepsilon\quad\forall\,\mathbf x,\mathbf y\in B_{r\varepsilon}(\hat{\mathbf x}),\]
\[|\mathbf x-\bar{\mathbf y}|\geq |\mathbf y-\bar{\mathbf y}|-|\mathbf x-\mathbf y|\geq 2\left(\frac{i_0}{\kappa}-r\varepsilon\right)-2r\varepsilon=\frac{2i_0}{\kappa}-4r\varepsilon\geq \frac{i_0}{\kappa}\quad\forall\,\mathbf x,\mathbf y\in B_{r\varepsilon}(\hat{\mathbf x}).\]

\end{proof}

To proceed, for $\zeta\in\Gamma_{\varepsilon,i_0}$
we define
\begin{equation}\label{deot}
T_\zeta=\frac{1}{2}\int_{\Pi}\zeta(\mathcal G\zeta-\lambda_\zeta x_2-\mu_\zeta) d\mathbf x.
\end{equation}

\begin{lemma}\label{lam4}
There exists some $C_2>0$ such that 
\begin{equation}\label{tc2u}
T_\zeta\leq C_2 \quad\forall\,\zeta\in\Gamma_{\varepsilon,i_0}.
\end{equation}
\end{lemma}

\begin{proof}

Fix $\zeta\in\Gamma_{\varepsilon,i_0}$.
Denote $\phi=\mathcal G\zeta-\lambda_\zeta x_2-\mu_\zeta$ and $\phi^+=\max\{\phi,0\}.$
By \eqref{lesss}, we have
\[V_\zeta=\{\mathbf x\in\Pi\mid \phi^+(\mathbf x)>0\},\]
from which we deduce that
\[T_\zeta=\frac{1}{2}\int_\Pi\zeta \phi d\mathbf x=\frac{1}{2}\int_\Pi\zeta\phi^+ d\mathbf x.\]
By \eqref{fvafn}, 
 $\phi^+$ vanishes on $\partial \Pi$. Therefore we can apply integration by parts to get
 \begin{equation}\label{ooh1}
T_\zeta=\frac{1}{2}\int_\Pi\zeta\phi^+ d\mathbf x=\frac{1}{2}\int_\Pi|\nabla \phi^+|^2 d\mathbf x.
\end{equation}
On the other hand,  by H\"older's inequality we have
\begin{equation}\label{ii33}
\begin{split}
 2T_\zeta&=\int_\Pi\zeta \phi^+ d\mathbf x \\
 & \leq \|\zeta\|_{L^p(\Pi)}\|\phi^+ \|_{L^q(\Pi)}  \\
& =\|\varrho^\varepsilon\|_{L^p(\mathbb R^2)}\|\phi^+ \|_{L^q(\Pi)}\\
&=\varepsilon^{-\frac{2}{q}}\|\varrho\|_{L^p(\mathbb R^2)} \|\phi^+ \|_{L^q(\Pi)}.
\end{split}
 \end{equation}
 Here $q=p/(p-1)$ is the H\"older conjugate exponent of $p.$ Since   $p>2,$ we have $q<2.$ We estimate $\|\phi^+\|_{L^q(\Pi)}$ as follows
  \begin{align}
  \|\phi^+ \|_{L^q(\Pi)}&\leq \|\mathbf 1_{V_\zeta}\|_{L^{\frac{2q}{2-q}}(\Pi)}  \|\phi^+ \|_{L^2(\Pi)} \label{dwh1}\\
  &\leq S\|\mathbf 1_{V_\zeta}\|_{L^{\frac{2q}{2-q}}(\Pi)}\|\nabla\phi^+ \|_{L^1(\Pi)} \label{dwh2}\\
  &\leq S\|\mathbf 1_{V_\zeta}\|_{L^{\frac{2q}{2-q}}(\Pi)}\|\mathbf 1_{V_\zeta}\|_{L^{2}(\Pi)}\|\nabla\phi^+ \|_{L^2(\Pi)} \label{dwh3}\\
  &=S\mathcal L^{1/q}(V_\zeta) \|\nabla\phi^+ \|_{L^2(\Pi)} \\
  &= S\pi^{1/q} r^{2/q}\varepsilon^{2/q}\|\nabla\phi^+ \|_{L^2(\Pi)}. \label{dwh4}
 \end{align}
 Here  $S$ is a generic positive constant.
Note that  we used H\"older's inequality in  \eqref{dwh1} and \eqref{dwh3},   Sobolev's inequality (see Theorem 4.8 in \cite{EV} for example)  in \eqref{dwh2},   and  the following fact in  \eqref{dwh4}  
\[
\mathcal L(V_\zeta)=\mathcal L\left( \{\mathbf x\in\mathbb R^2\mid \varrho^\varepsilon(\mathbf x)>0\}\right)=\varepsilon^2\mathcal L\left( \{\mathbf x\in\mathbb R^2\mid \varrho(\mathbf x)>0\}\right)=\pi r^2\varepsilon^2.\]
Therefore we obtain
\begin{equation}\label{rhzo}
T_\zeta\leq \frac{1}{2}{S\pi^{1/q} r^{2/q}\|\varrho\|_{L^p(\mathbb R^2)}\|\nabla\phi^+ \|_{L^2(\Pi)}} .
\end{equation}
 The desired result follows from \eqref{ooh1} and \eqref{rhzo} immediately.

\end{proof}

\begin{remark}
From \eqref{ooh1}, it is easy to see that
\[T_\zeta=\frac{1}{2}\int_{V_\zeta}|\nabla\mathcal G\zeta|^2d\mathbf x.\]
Therefore  $T_\zeta$ in fact  represents the kinetic energy of the fluid on the vortex core $V_\zeta$.

\end{remark}

\begin{lemma}\label{lam6}
There exists some $C_3>0$ such that
\begin{equation}\label{tc20u}
\mu_\zeta \geq -\frac{\kappa}{2\pi}\ln\varepsilon-\frac{\lambda_\zeta i_0}{\kappa}-C_3 \quad\forall\,\zeta\in\Gamma_{\varepsilon,i_0}.
\end{equation}
\end{lemma}

\begin{proof}
By the definition of $T_\zeta$ \eqref{deot}, it holds that
\[T_\zeta=E(\zeta)-\frac{1}{2}\lambda_\zeta i_0-\frac{1}{2}\kappa\mu_\zeta. \]
Then \eqref{tc20u} follows from \eqref{tc1u} and   \eqref{tc2u}.
\end{proof}

\begin{lemma}\label{lam7}
There exists some $C_4>0$ such that
\begin{equation}\label{tc3u}
\mathcal G\zeta(\mathbf x)\geq -\frac{\kappa}{2\pi}\ln\varepsilon+\lambda_\zeta x_2- \frac{\lambda_\zeta i_0 }{\kappa} -C_4\quad\forall\,\zeta\in\Gamma_{\varepsilon,i_0},\,\,\mathbf x\in V_\zeta.
\end{equation}

\end{lemma}
\begin{proof}
For any $\zeta\in\Gamma_{\varepsilon,i_0}$ and  $\mathbf x\in V_\zeta,$ recalling  \eqref{lesss}, we have 
\[ \mathcal G\zeta(\mathbf x)-\lambda_\zeta x_2 > \mu_\zeta.\]
Then \eqref{tc3u} follows from \eqref{tc20u}.
\end{proof}

 For  $\zeta\in\Gamma_{\varepsilon,i_0}$, define
 \begin{equation}\label{u1s1}
 F_\zeta(s)=\int_{-\infty}^sf_\zeta(\tau)d\tau=\int_{\mu_\zeta}^sf_\zeta(\tau)d\tau, 
\end{equation}
where $f_\zeta$ is determined by Proposition \ref{prop1}.
Then  $F_\zeta$ is locally Lipschitz continuous in $(-\infty,\sigma_\zeta)$ and \[F_\zeta'(s)=f_\zeta(s)\mbox{  a.e.  }s\in(-\infty,\sigma_\zeta).\] Recall that $\sigma_\zeta$ is defined by \eqref{fvasg}. 
Moreover, since $F_\zeta$ vanishes on $(-\infty,\mu_\zeta],$ taking into account  
 \eqref{lesss}, we have
\begin{equation}\label{bdsp}
\{\mathbf x\in \Pi\mid F_\zeta(\mathcal G\zeta(\mathbf x)-\lambda_\zeta x_2)>0\}\subset \{\mathbf x\in \Pi\mid \mathcal G\zeta(\mathbf x)-\lambda_\zeta x_2>\mu_{\zeta}\}=V_\zeta.
\end{equation}
Consequently  $F_\zeta(\mathcal G\zeta-\lambda_\zeta x_2)$ has bounded support.

Our next step is to prove the uniform boundedness from above for the traveling speed $\lambda_\zeta$. To this end, we need the following Pohozaev type identity, which is a straightforward consequence of Lemma 9 in \cite{B11}. See also (13) in \cite{B10}, p.565.   
\begin{lemma}\label{lam79}
For any $\zeta\in \Gamma_{\varepsilon,i_0},$ the following identity holds
\begin{equation}\label{vv1}
\int_\Pi F_\zeta(\mathcal G\zeta-\lambda_\zeta x_2)d\mathbf x=\frac{\lambda_\zeta i_0}{2}.
\end{equation}
\end{lemma}

Based on Lemmas \ref{lam4} and   \ref{lam79}, we are able to prove

\begin{lemma}\label{lam8}
There exists some $C_5>0$ such that
\begin{equation}\label{tsu0}
\lambda_\zeta\leq C_5\quad\forall\, \zeta\in\Gamma_{\varepsilon,i_0}.
\end{equation}
\end{lemma}
\begin{proof}
Fix $\zeta\in\Gamma_{\varepsilon,i_0}$. Since $f_\zeta$ is increasing, we have
\[F_\zeta(s)=\int_{\mu_{\zeta}}^sf_\zeta(\tau)d\tau\leq \int_{\mu_{\zeta}}^sf_\zeta(s)d\tau=f_\zeta(s)(s-\mu_\zeta)\quad\forall\,s\in\mathbb R.\]
Therefore
\begin{equation}\label{cvv}
\begin{split}
\int_\Pi F_\zeta(\mathcal G\zeta-\lambda_\zeta x_2)d\mathbf x&\leq \int_\Pi f_\zeta(\mathcal G\zeta-\lambda_\zeta x_2)(\mathcal G\zeta-\lambda_\zeta x_2-\mu_\zeta)d\mathbf x\\
&=\int_\Pi \zeta(\mathcal G\zeta-\lambda_\zeta x_2-\mu_\zeta)d\mathbf x\\
&=2T_\zeta.
\end{split}
\end{equation}
The desired estimate \eqref{tsu0} follows immediately from \eqref{tc2u} and \eqref{vv1}.

\end{proof}

With Lemma \ref{lam8}, we can improve Lemma \ref{lam7} as follows.

\begin{lemma}\label{lam9}
There exists some $C_6>0$ such that
\begin{equation}\label{lzi}
\mathcal G\zeta(\mathbf x)\geq -\frac{\kappa}{2\pi}\ln\varepsilon+\lambda_\zeta x_2-C_6\quad \forall\,\zeta\in\Gamma_{\varepsilon,i_0},\,\,\mathbf x\in V_\zeta.
\end{equation}
Consequently (recall $\lambda_\zeta>0$)  
\begin{equation}\label{lziii}
\mathcal G\zeta(\mathbf x)\geq -\frac{\kappa}{2\pi}\ln\varepsilon-C_6\quad \forall\,\zeta\in\Gamma_{\varepsilon,i_0},\,\,\mathbf x\in V_\zeta.
\end{equation}
\end{lemma}

Now we turn to estimating the size of the vortex core $V_\zeta$. As a preliminary, we need the following a priori upper bound for $\mathcal G\zeta$ based on Lemma \ref{lemm2}.

\begin{lemma}\label{lam10}
There exists some $C_7>0$ such that
\begin{equation}\label{lzu}
\mathcal G\zeta(\mathbf x)\leq   C_7  (x_2\varepsilon)^{-1/2}\quad \forall\,\zeta\in\Gamma_{\varepsilon,i_0}, \,\, x_2\geq 1.
\end{equation}
\end{lemma}

\begin{proof}
Fix $\zeta\in\Gamma_{\varepsilon,i_0}.$
Choosing $s=2 $ and $\theta=1/2$ in  Lemma \ref{lemm2},  it holds  for any $\mathbf x=(x_1,x_2)\in\Pi$ with $x_2\geq 1$ that 
\begin{align*}
\mathcal G\zeta(\mathbf x)&\leq x_2^{-1}(c_1\ln x_2+c_2)I(\zeta)+c_3x_2^{-1/2}\|\zeta\|_{L^2(\Pi)}^{1/2}I(\zeta)^{1/2}\\
&=x_2^{-1}(c_1\ln x_2+c_2)i_0+c_3x_2^{-1/2}\|\varrho^\varepsilon\|_{L^2(\Pi)}^{1/2}i_0^{1/2}\\
&=\frac{i_0(c_1\ln x_2+c_2)}{x_2}+\frac{c_3\|\varrho\|^{1/2}_{L^2(\Pi)}i_0^{1/2}}{\sqrt{x_2\varepsilon}}.
\end{align*}
Here we used the fact that $\|\varrho^\varepsilon\|_{L^2(\Pi)}=\varepsilon^{-1}\|\varrho\|_{L^2(\Pi)}$.
This obviously implies \eqref{lzu}.

\end{proof}

The following lemma provides a rough estimate for the size of  $V_\zeta$ in the $x_2$ direction.

\begin{lemma}\label{lam99}
There exists some $C_8>0$ such that
\begin{equation}\label{lzw}
x_2\leq C_8\varepsilon^{-1} \quad\forall\,\zeta\in\Gamma_{\varepsilon,i_0},\,\,\mathbf x\in V_\zeta.
\end{equation}
\end{lemma}
\begin{proof}
 Without loss of generality, we assume that $\varepsilon_0$ is small enough such that 
\[-\frac{\kappa}{2\pi}\ln\varepsilon-C_6\geq 1\quad\forall\,\varepsilon\in(0,\varepsilon_0).\]
Hence by  \eqref{lziii} we have 
\begin{equation}\label{lziiii}
\mathcal G\zeta(\mathbf x)\geq 1\quad \forall\,\zeta\in\Gamma_{\varepsilon,i_0},\,\,\mathbf x\in V_\zeta.
\end{equation}
Then the desired result follows from \eqref{lzu}.
\end{proof}

\begin{remark}
In the proof of Lemma \ref{lam10}, we could have chosen other $p$ and $\theta$ to get a better estimate than \eqref{lzu}, hence a better estimate than  \eqref{lzw}. However,  as we will see, this makes no difference to the subsequent arguments.
\end{remark}

To estimate the size of $V_\zeta$ in the $x_1$ direction, we need the following lemma similar to Lemma \ref{lam10}.
Recall the definition of $\Gamma^0_{\varepsilon,i_0}$  given by \eqref{gam0}.

\begin{lemma}\label{lam100}
There exists some $C_9>0$ such that 
\begin{equation}\label{xi1}
\mathcal G\zeta(\mathbf x)\leq C_9\varepsilon^{2/p-3}|x_1|^{-1/(2p)}\quad \forall\,\zeta\in \Gamma_{\varepsilon,i_0}^0, \,\,\mathbf x\in V_\zeta,\,\, |x_1|\geq 1.
\end{equation}

\end{lemma}
\begin{proof}
Fix $\zeta\in\Gamma^0_{\varepsilon,i_0}$.  By Remark \ref{rmk1},  $\zeta$ is Steiner-symmetric in the $x_2$-axis. Choosing $s=p$ in Lemma \ref{lemm3}, it holds for any $\mathbf x=(x_1,x_2)\in \Pi$ with $|x_1|\geq 1$ that 
\begin{align*}
\mathcal G\zeta(\mathbf x)&\leq K\left(i_0+\|\zeta\|_{L^1(\Pi)}+\|\zeta\|_{L^p(\Pi)}\right)x_2|x_1|^{-1/(2p)}\\
&= K \left(i_0+\kappa+\varepsilon^{ {2}/{p}-2}\|\varrho\|_{L^p(\Pi)}\right)x_2|x_1|^{-1/(2p)}\\
&\leq KC_8\varepsilon^{-1} \left(i_0+\kappa+\varepsilon^{ {2}/{p}-2}\|\varrho\|_{L^p(\Pi)}\right)|x_1|^{-1/(2p)}\\
&=KC_8\varepsilon^{2/p-3} \left((i_0+\kappa)\varepsilon^{ 2-{2}/{p}}+\|\varrho\|_{L^p(\Pi)}\right)|x_1|^{-1/(2p)}\\
&\leq KC_8\varepsilon^{2/p-3} \left((i_0+\kappa)\varepsilon_0^{ 2-{2}/{p}}+\|\varrho\|_{L^p(\Pi)}\right)|x_1|^{-1/(2p)}.
\end{align*}
Here we used Lemma \ref{lam99} and the  following facts: 
\[\|\zeta\|_{L^1(\Pi)}=\|\varrho^\varepsilon\|_{L^1(\Pi)}=\kappa,\]
\[\|\zeta\|_{L^p(\Pi)}=\|\varrho^\varepsilon\|_{L^p(\Pi)}=\varepsilon^{2/p-2}\|\varrho\|_{L^p(\Pi)}.\]
  Hence the proof is finished by choosing 
 \[C_9=KC_8  \left((i_0+\kappa)\varepsilon_0^{ 2-{2}/{p}}+\|\varrho\|_{L^p(\Pi)}\right).\]
\end{proof}

Based on  Lemma \ref{lam9} and Lemma \ref{lam100}, we are able to provide a  rough estimate for the size of  $V_\zeta$ in the $x_1$ direction.

\begin{lemma}\label{lam101}
There exists some $C_{10}>0$ such that 
\begin{equation}\label{xi13}
|x_1|\leq C_{10}\varepsilon^{4-6p} \quad\forall\,\zeta\in\Gamma^0_{\varepsilon,i_0},\,\,\mathbf x\in V_\zeta.
\end{equation}
\end{lemma}

\begin{proof}
Fix $\zeta\in\Gamma^0_{\varepsilon,i_0}$.  By \eqref{lziii} and \eqref{xi1},  for any $\mathbf x\in V_\zeta$ such that $|x_1|\geq 1$, we have
\[-\frac{\kappa}{2\pi}\ln\varepsilon-C_6 \leq C_9\varepsilon^{2/p-3}|x_1|^{-1/(2p)},\]
which implies
\[-\frac{\kappa}{2\pi}\ln\varepsilon_0-C_6 \leq C_9\varepsilon^{2/p-3}|x_1|^{-1/(2p)}.\]
Without loss of generality, assume that $\varepsilon_0$ is small enough such that 
\[-\frac{\kappa}{2\pi}\ln\varepsilon_0-C_6>0.\] 
Then  
\[|x_1|\leq {C_9^{2p}}\left(-\frac{\kappa}{2\pi}\ln\varepsilon_0-C_6\right)^{-2p}\varepsilon^{4-6p}.\]
Hence for any $\mathbf x\in V_\zeta$ we have 
\begin{align*}
|x_1|&\leq {C_9^{2p}}\left(-\frac{\kappa}{2\pi}\ln\varepsilon_0-C_6\right)^{-2p}\varepsilon^{4-6p}+1\\
&\leq {C_9^{2p}}\left(-\frac{\kappa}{2\pi}\ln\varepsilon_0-C_6\right)^{-2p}\varepsilon^{4-6p}+\varepsilon_0^{6p-4}\varepsilon^{4-6p}.
\end{align*}
The proof is completed by choosing 
\[C_{10}={C_9^{2p}}\left(-\frac{\kappa}{2\pi}\ln\varepsilon_0-C_6\right)^{{-2p}}+\varepsilon_0^{6p-4}.\]
\end{proof}

For $\zeta\in \Gamma_{\varepsilon,i_0}$, define
\[l_\zeta=\sup_{\mathbf x,\mathbf y\in  V_\zeta}|\mathbf x-\bar{\mathbf y}|.\]
Since $V_\zeta$ is contained in the upper half-plane, it is easy to check that 
\begin{equation}\label{due7}
{\rm diam}(V_\zeta)\leq l_{\zeta}.
\end{equation}

We have the following rough estimate for $l_\zeta$.
\begin{lemma}\label{lam201}
There exists some $C_{11}>0$ such that 
\begin{equation}
l_\zeta\leq C_{11}\varepsilon^{4-6p}\quad \forall\,\zeta\in \Gamma_{\varepsilon,i_0}.\end{equation}

\end{lemma}

\begin{proof}
Fix $\zeta\in \Gamma_{\varepsilon,i_0}.$
Without loss of generality, assume that $\zeta\in\Gamma^0_{\varepsilon,i_0}$.   Combining Lemma \ref{lam99} and Lemma \ref{lam101}, it holds  for any $\mathbf x,\mathbf y\in V_\zeta$ that
\begin{align*}
|\mathbf x-\bar{\mathbf y}|^2&=\left(x_1-y_1\right)^2+(x_2+y_2)^2\\
&\leq 2\left(x_1^2+y_1^2\right)+2\left(x_2^2+y_2^2\right)\\
&\leq 4\left(C^2_{10}\varepsilon^{8-12p}+C_8^2\varepsilon^{-2}\right)\\
&\leq 4\left(C^2_{10} +C_8^2\varepsilon_0^{12p-10}\right)\varepsilon^{8-12p}.
\end{align*}
The proof is finished by choosing  
\[C_{11}=2\left(C^2_{10} +C_8^2\varepsilon_0^{12p-10}\right)^{1/2}.\]
\end{proof}

Having made enough preparations, we are ready to deduce the most important uniform estimate in this paper, that is, the uniform boundedness of  $l_\zeta$, from which the uniform boundedness of $V_\zeta $ follows immediately due to \eqref{due7}

\begin{proposition}\label{lam202}
There exists some $C_{12}>0$ such that 
\begin{equation}
l_\zeta\leq C_{12}\quad \zeta\in\Gamma_{\varepsilon,i_0}.
\end{equation}

\end{proposition}

\begin{proof}
We prove this proposition by contradiction. 
Assume that there exist $\{\varepsilon_n\}_{n=1}^{+\infty}\subset(0,\varepsilon_0)$ and  $\{\zeta_n\}_{n=1}^{+\infty}\subset\Gamma_{\varepsilon_n,i_0}$ such that
\begin{equation}\label{sco1}
\lim_{n\to+\infty}l_{\zeta_n}=+\infty.
\end{equation}
By Lemma \ref{lam201} we must have
\begin{equation}\label{sco2}
\lim_{n\to+\infty}\varepsilon_n=0.
\end{equation}
Without loss of generality, we assume that   $\{\zeta_n\}_{n=1}^{+\infty}\subset \Gamma^0_{\varepsilon,i_0}$.

Below for simplicity denote $l_n=l_{\zeta_n}, V_n=V_{\zeta_n}.$
For any $\mathbf x\in V_n$, by \eqref{lziii} it holds that
\[\mathcal G\zeta_n(\mathbf x)\geq -\frac{\kappa}{2\pi}\ln\varepsilon_n-C_6,\]
which can be equivalently written as
\begin{equation}\label{cbdd}
\int_\Pi\ln\frac{\varepsilon_n}{|\mathbf x-\mathbf y|}\zeta_n(\mathbf y)d\mathbf y\geq -\int_{V_n}\ln|\mathbf x-\bar{\mathbf y}|\zeta_n(\mathbf y)d\mathbf y-2\pi C_6.
\end{equation}
Since $|\mathbf x-\bar{\mathbf y}|\leq l_n$ for any $\mathbf x,\mathbf y\in V_n$,
we further get
\begin{equation}\label{rp1}
\int_\Pi\ln\frac{\varepsilon_n}{|\mathbf x-\mathbf y|}\zeta_n(\mathbf y)d\mathbf y\geq   -\kappa\ln l_n-2\pi C_6.
\end{equation}
Divide the integral in \eqref{rp1} into two parts
\begin{align*} 
\int_\Pi\ln\frac{\varepsilon_n}{|\mathbf x-\mathbf y|}\zeta_n(\mathbf y)d\mathbf y&= \int_{|\mathbf x-\mathbf y|\leq \frac{l_n}{N}}\ln\frac{\varepsilon_n}{|\mathbf x-\mathbf y|}\zeta_n(\mathbf y)d\mathbf y+
\int_{|\mathbf x-\mathbf y|\geq \frac{l_n}{N}}\ln\frac{\varepsilon_n}{|\mathbf x-\mathbf y|}\zeta_n(\mathbf y)d\mathbf y\\
&:=I_1+I_2,
\end{align*}
where $N$ is a fixed large positive integer such that
\begin{equation}\label{nbss}
N>60p-30.
\end{equation}
Note that since $p>2$, we have
\begin{equation}\label{nbss2}
N>90.
\end{equation}
For $I_1$, it is easy to check by change of variables and the rearrangement inequality  \eqref{rri1}  that
\begin{equation}\label{rp2}
\int_{|\mathbf x-\mathbf y|\leq \frac{l_n}{N}}\ln\frac{\varepsilon_n}{|\mathbf x-\mathbf y|}\zeta_n(\mathbf y)d\mathbf y\leq \int_{\Pi}\ln\frac{\varepsilon_n}{|\mathbf x-\mathbf y|}\zeta_n(\mathbf y)d\mathbf y\leq \int_{\mathbb R^2}\ln\frac{1}{|\mathbf y|}\varrho (\mathbf y)d\mathbf y\leq K_1
\end{equation}
for some positive number  $K_1$   depending only on $\varrho$.
For $I_2$, we have
\begin{equation}\label{rp3}
\int_{|\mathbf x-\mathbf y|\geq \frac{l_n}{N}}\ln\frac{\varepsilon_n}{|\mathbf x-\mathbf y|}\zeta_n(\mathbf y)d\mathbf y\leq \ln\frac{N\varepsilon_n}{l_n}\int_{|\mathbf x-\mathbf y|\geq \frac{l_n}{N}}\zeta_n(\mathbf y)d\mathbf y.
\end{equation}
From  \eqref{rp1}, \eqref{rp2} and \eqref{rp3} we can get
\begin{equation}\label{y1p}
\ln\frac{N\varepsilon_n}{l_n}\int_{|\mathbf x-\mathbf y|\geq \frac{l_n}{N}}\zeta_n(\mathbf y)d\mathbf y\geq  -\kappa \ln l_n-2\pi C_6-K_1.
\end{equation}
Recalling \eqref{sco1} and \eqref{sco2},  it holds that  for sufficiently large $n$ that
\[\frac{N\varepsilon_n}{l_n}<1,\]
 therefore \eqref{y1p} is equivalent to
\begin{equation}\label{y2p}
\int_{|\mathbf x-\mathbf y|\geq \frac{l_n}{N}}\zeta_n(\mathbf y)d\mathbf y\leq  \frac{ \kappa \ln l_n+2\pi C_6+K_1}{ \ln l_n-\ln\varepsilon_n-\ln N}.
\end{equation}
Define 
\begin{align*}
h(s)=\frac{ \kappa s+2\pi C_6+K_1}{ s-\ln\varepsilon_n-\ln N}.
\end{align*}
It is easy to check that $h$ is  increasing in $(0,+\infty)$ if $n$ is large enough. Combining  Lemma \ref{lam201}, we get from \eqref{y2p} that
\begin{equation}\label{y3p}
\int_{|\mathbf x-\mathbf y|\geq \frac{l_n}{N}}\zeta_n(\mathbf y)d\mathbf y\leq  \frac{ \kappa \ln \left(C_{11}\varepsilon_n^{4-6p}\right)+2\pi C_6+K_1}{ \ln \left(C_{11}\varepsilon_n^{4-6p}\right)-\ln\varepsilon_n-\ln N}.
\end{equation}
Due to \eqref{sco2}, we have
 \begin{equation}\label{codnt1}
 \lim_{n\to+\infty}\frac{ \kappa \ln \left(C_{11}\varepsilon_n^{4-6p}\right)+2\pi C_6+K_1}{ \ln \left(C_{11}\varepsilon_n^{4-6p}\right)-\ln\varepsilon_n-\ln N}=\frac{6p-4}{6p-3}\kappa.
 \end{equation}
 Recalling \eqref{nbss}, it is easy to check that 
 \begin{equation}\label{codnt}\frac{6p-4}{6p-3} < 1-\frac{10}{N}.\end{equation}
Therefore by \eqref{y3p}-\eqref{codnt} we have
\[
\int_{|\mathbf x-\mathbf y|\geq \frac{l_n}{N}}\zeta_n(\mathbf y)d\mathbf y< \left(1-\frac{10}{N}\right)\kappa,
\]
provided that $n$ is large enough.
Taking into account the fact that $\|\zeta_n\|_{L^1(\Pi)}=1 \,\,\forall\, n$, we get for sufficiently large $n$ that 
\begin{equation}\label{cont}
\int_{|\mathbf x-\mathbf y|\leq \frac{l_n}{N}}\zeta_n(\mathbf y)d\mathbf y>\frac{10}{N}\kappa.
\end{equation}
Note that \eqref{cont} holds for arbitrary $\mathbf x\in V_n$.

Based on \eqref{cont},
 we  deduce a contradiction as follows.  Denote
\[ s_n=\sup_{\mathbf x\in V_n}x_2\quad t_n=\sup_{\mathbf x\in V_n}|x_1|.\]
Since we have assumed $\{\zeta_n\}_{n=1}^{+\infty}\subset \Gamma^0_{\varepsilon,i_0}$, is easy to check that 
\begin{equation}\label{gele}
l_n\leq 2(s_n+t_n)\quad \forall\,n,
\end{equation}
which yields   either $s_n\geq l_n/4$ or  $t_n\geq l_n/4$ for each $n$. Therefore  at least one of the following two assertions holds.
\begin{itemize}
\item [ (A1)] There exists a subsequence   $\{s_{n_j}\}_{j=1}^{+\infty}$ such that 
\[s_{n_j}\geq \frac{l_{n_j}}{4}\quad \forall\,j.\]
\item [(A2)] There exists a subsequence   $\{t_{n_j}\}_{j=1}^{+\infty}$ such that 
\[t_{n_j}\geq \frac{l_{n_j}}{4}\quad \forall\,j.\]

\end{itemize}

First we show that (A1) is impossible. In fact, if (A1) holds, then  we can choose  a sequence $\{\mathbf x_j\}_{j=1}^{+\infty}\subset  V_{n_j}$ such that 
\begin{equation}
 x_{j,2}\geq \frac{l_{n_j}}{8}\quad \forall\,j.  
\end{equation}
Here $\mathbf x_j=(x_{j,1},x_{j,2}).$
Using \eqref{cont}, we have  for sufficiently large $n$ that 
\begin{align*}
\int_\Pi y_2\zeta_{n_j}(\mathbf y)d\mathbf y
&\geq \int_{|\mathbf x_j-\mathbf y|\leq \frac{l_{n_j}}{N}}y_2\zeta_{n_j}(\mathbf y)d\mathbf y\\
&\geq \int_{|\mathbf x_n-\mathbf y|\leq \frac{l_{n_j}}{N}}\left(x_{j,2}-|x_{j,2}-y_2|\right)\zeta_{n_j}(\mathbf y)d\mathbf y\\
&\geq \int_{|\mathbf x_n-\mathbf y|\leq \frac{l_{n_j}}{N}}\left(\frac{l_{n_j}}{8}-\frac{l_{n_j}}{N}\right)\zeta_{n_j}(\mathbf y)d\mathbf y\\
&\geq \left(\frac{1}{8}-\frac{1}{N}\right)\frac{10}{N}\kappa  l_{n_j},
\end{align*}
which goes to $+\infty$ as $n\to+\infty.$ This contradicts the fact that $I(\zeta_n)=i_0$. Therefore (A1) is impossible.

To finish the proof, it  suffices to show that (A2) cannot possibly hold, either. In fact, if (A2) holds, then  for any $\mathbf x=(x_1,x_2)\in V_{n_j}$ it holds that 
\begin{equation}\label{monee}
\int_{|\mathbf x-\mathbf y|\leq \frac{l_{n_j}}{N}}\zeta_{n_j}(\mathbf y)d\mathbf y\leq \int_{|x_1-y_1|\leq \frac{l_{n_j}}{N}}\zeta_{n_j}(\mathbf y)d\mathbf y\leq \int_{|x_1-y_1|\leq \frac{4t_{n_j}}{N}}\zeta_{n_j}(\mathbf y)d\mathbf y.
\end{equation}
This together with \eqref{cont} yields for sufficiently large $j$ that 
\begin{equation}\label{csdfd}
\int_{|x_1-y_1|\leq \frac{4t_{n_j}}{N}}\zeta_{n_j}(\mathbf y)d\mathbf y>\frac{10}{N}\kappa\quad\forall\, \mathbf x\in V_{n_j}.
\end{equation}
Below fix a sufficient large $j$ such that \eqref{csdfd} holds.
Using the fact that each $\zeta_{n_j}$ is Steiner-symmetric in the $x_2$-axis, we obtain
\begin{equation}\label{con11}
\int_{|a-y_1|\leq \frac{4t_{n_j}}{N}}\zeta_{n_j}(\mathbf y)d\mathbf y>\frac{10}{N}\kappa\quad\forall\, a\in(-t_{n_j},t_{n_j}).
\end{equation}
For any $a\in\mathbb R$, denote by $\lceil a\rceil$ the smallest integer not less than $a$. Define 
\[a_k=\frac{4(2k-1)t_{n_j}}{N},\quad k=1,\cdot\cdot\cdot,\left\lceil\frac{N}{10}\right\rceil.\]
Then for any $1\leq k\leq \left\lceil\frac{N}{10}\right\rceil,$ it holds that
\[0<a_k  \leq \frac{ 8\left\lceil\frac{N}{10}\right\rceil-4 }{N}t_{n_j}\leq \frac{ 8\left(\frac{N}{10}+1\right)-4 }{N}t_{n_j}=\left(\frac{4}{5}+\frac{4}{N}\right)t_{n_j}<t_{n_j}. \]
Here we used \eqref{nbss2} in the last inequality.
Choosing $a=a_k$  in \eqref{con11} and summing over $k=  1,\cdot\cdot\cdot, \left\lceil\frac{N}{10}\right\rceil$, we get
\begin{equation}\label{coniii}
\int_\Pi\zeta_{n_j}(\mathbf y)d\mathbf y\geq \sum_{k=1}^{\left\lceil\frac{N}{10}\right\rceil}\int_{|y_1-a_k|\leq \frac{4t_{n_j}}{N}}\zeta_{n_j}(\mathbf y)d\mathbf y>\sum_{k=1}^{\left\lceil\frac{N}{10}\right\rceil}\frac{10}{N}\kappa\geq \kappa,
\end{equation}
a contradiction to the fact that $\|\zeta_{n_j}\|_{L^1(\Pi)}=\kappa.$ Hence the proof is finished.

\end{proof}

\subsection{Asymptotic behavior}
In this section, we study the asymptotic behavior of the maximizers as $\varepsilon\to0^+$. 
With Proposition \ref{lam202} at hand, the unboundedness of the upper half-plane does not cause trouble anymore, hence the subsequent proofs are analogous to the ones in \cite{T}.

First we prove a better  bound for the size of the vortex core.

\begin{proposition}\label{lam203}
There exists some $C_{13}>0$ such that 
\[{\rm diam}(V_\zeta)\leq C_{13}\varepsilon\quad \forall\,\zeta\in\Gamma_{\varepsilon,i_0}.\]

\end{proposition}

\begin{proof}
Fix $\zeta\in\Gamma_{\varepsilon,i_0}$. By \eqref{lziii} it holds for any $\mathbf x\in V_\zeta$ that 
\[\mathcal G\zeta (\mathbf x)\geq -\frac{\kappa}{2\pi}\ln\varepsilon -C_6,\]
or equivalently,
\begin{equation}\label{twis}
\int_\Pi\ln\frac{\varepsilon }{|\mathbf x-\mathbf y|}\zeta(\mathbf y)d\mathbf y\geq -\int_{V_\zeta}\ln|\mathbf x-\bar{\mathbf y}|\zeta(\mathbf y)d\mathbf y-2\pi C_6.
\end{equation}
Taking into account Proposition \ref{lam202}, we  have 
\begin{equation}\label{twis3}
\int_\Pi\ln\frac{\varepsilon }{|\mathbf x-\mathbf y|}\zeta(\mathbf y)d\mathbf y\geq -\kappa\ln C_{12}-2\pi C_6.
\end{equation}
As in the proof of Proposition \ref{lam202}, we divide the integral in \eqref{twis3} into two parts
\begin{align*}
\int_\Pi\ln\frac{\varepsilon }{|\mathbf x-\mathbf y|}\zeta(\mathbf y)d\mathbf y&=\int_{|\mathbf x-\mathbf y|\leq R\varepsilon}\ln\frac{\varepsilon }{|\mathbf x-\mathbf y|}\zeta(\mathbf y)d\mathbf y+\int_{|\mathbf x-\mathbf y|\geq R\varepsilon}\ln\frac{\varepsilon }{|\mathbf x-\mathbf y|}\zeta(\mathbf y)d\mathbf y,
\end{align*}
where $R>1$ is  to be determined later.
As in \eqref{rp2},  for the first integral  it holds that  
\begin{equation}\label{may1}
\int_{|\mathbf x-\mathbf y|\leq R\varepsilon}\ln\frac{\varepsilon }{|\mathbf x-\mathbf y|}\zeta(\mathbf y)d\mathbf y \leq \int_{\Pi}\ln\frac{\varepsilon}{|\mathbf x-\mathbf y|}\zeta(\mathbf y)d\mathbf y\leq \int_{\mathbb R^2}\ln\frac{1}{|\mathbf y|}\varrho (\mathbf y)d\mathbf y\leq K_1,
\end{equation}
 where $K_1>0$ depends only on $\varrho$. For the second integral, we have
\begin{equation}\label{may2}
\int_{|\mathbf x-\mathbf y|\geq R\varepsilon}\ln\frac{\varepsilon }{|\mathbf x-\mathbf y|}\zeta(\mathbf y)d\mathbf y\leq -\ln R\int_{|\mathbf x-\mathbf y|\geq R\varepsilon}\zeta(\mathbf y)d\mathbf y.
\end{equation}
Combining \eqref{twis3}, \eqref{may1} and \eqref{may2} we obtain
\begin{equation}\label{omi1}
\int_{|\mathbf x-\mathbf y|\geq R\varepsilon}\zeta(\mathbf y)d\mathbf y\leq \frac{\kappa\ln C_{12}+2\pi C_6+K_1}{\ln R}.
\end{equation}
Choose  $R$ to be sufficiently large, depending only on $\varrho$ and $i_0$, such that
\[\frac{\kappa\ln C_{12}+2\pi C_6+K_1}{\ln R}\leq \frac{\kappa}{3}.\]
Then we get from \eqref{omi1} that 
\[\int_{|\mathbf x-\mathbf y|\geq R\varepsilon}\zeta(\mathbf y)d\mathbf y\leq \frac{\kappa}{3},\]
or equivalently,
\begin{equation}\label{may5}
\int_{|\mathbf x-\mathbf y|\leq R\varepsilon}\zeta(\mathbf y)d\mathbf y\geq \frac{2}{3}\kappa.
\end{equation}
Since \eqref{may5} holds for arbitrary $\mathbf x\in V_\zeta,$   we can easily show that ${\rm diam}(V_\zeta)\leq 2R\varepsilon.$ In fact, suppose otherwise ${\rm diam}(V_\zeta)>2 R\varepsilon.$ Then we can choose two points $\mathbf x_1,\mathbf x_2\in V_\zeta$ such that $|\mathbf x_1-\mathbf x_2|>2R\varepsilon.$ It is clear that $B_{R\varepsilon}(\mathbf x_1)\cap B_{R\varepsilon}(\mathbf x_2)=\varnothing$, therefore
\[\int_\Pi\zeta(\mathbf y)d\mathbf y\geq \int_{|\mathbf x_1-\mathbf y|\leq R\varepsilon}\zeta(\mathbf y)d\mathbf y+\int_{|\mathbf x_2-\mathbf y|\leq R\varepsilon}\zeta(\mathbf y)d\mathbf y\geq \frac{4}{3}\kappa,\]
a  contradiction to the fact that $\|\zeta\|_{L^1(\Pi)}=\kappa.$ The proof is finished by choosing $C_{13}=2R.$
\end{proof}

For any $\zeta\in\Gamma^0_{\varepsilon,i_0}$ its center is $\hat{\mathbf x}=(0,i_0/\kappa).$
From this fact and  Proposition \ref{lam203}, we immediately obtain

  \begin{lemma}\label{cids}
Any   $\zeta\in \Gamma^0_{\varepsilon,i_0}$ converges uniformly to   $\kappa\delta_{\hat{\mathbf x}}$ in the distributional sense as $\varepsilon\to0^+$. More precisely,  for any $\varphi\in C(\Pi)$ and   any   $\epsilon>0,$ there exists some $\delta>0,$ depending only on $\varrho,i_0,\varphi$ and $\epsilon$, such that
 \[\left|\int_\Pi \varphi \zeta d\mathbf x-\kappa\varphi(\hat{\mathbf x}) \right|<\epsilon\quad\forall\,\zeta\in\Gamma_{\varepsilon,i_0}^0.\]

 \end{lemma}

  To study the limit of the traveling speed $\lambda_\zeta$ as $\varepsilon\to0^+,$ we prove a  useful formula below.
\begin{lemma}\label{ffs1}
For any $\zeta\in\Gamma_{\varepsilon,i_0},$ it holds that
\[\lambda_\zeta=\frac{1}{2\pi\kappa}\int_\Pi\int_\Pi\frac{x_2+y_2}{|\mathbf x-\bar{\mathbf y}|^2}\zeta(\mathbf x)\zeta(\mathbf y)d\mathbf xd\mathbf y.\]
\end{lemma}
\begin{proof}
Fix $\zeta\in\Gamma_{\varepsilon,i_0}$.  Recall that  $F_\zeta$ is defined by \eqref{u1s1}. By an almost identical argument as in the proof of Lemma 9 in \cite{B11}, we can show that  $F_\zeta(\mathcal G\zeta-\lambda_\zeta x_2)\in W^{1,p}(\Pi)$ and, moreover,  the following chain rule  holds true
\begin{equation}\label{crl}
 {\partial}_{x_2}\left(F_\zeta(\mathcal G\zeta-\lambda_\zeta x_2)\right)=f_\zeta(\mathcal G\zeta-\lambda_\zeta x_2) {\partial}_{x_2}(\mathcal G\zeta-\lambda_\zeta x_2)=\zeta(\partial_{x_2}\mathcal G\zeta-\lambda_\zeta).
 \end{equation}
Integrating both sides of \eqref{crl} over $\Pi$,  and taking into account the fact that $F_\zeta(\mathcal G\zeta-\lambda_\zeta x_2)$ has bounded support, we get
\begin{equation}\label{crll}
 \lambda_\zeta=\frac{1}{\kappa}\int_\Pi\zeta\partial_{x_2}\mathcal G\zeta d\mathbf x.
 \end{equation}
On the other hand, by a direct calculation we have  
 \begin{equation}\label{c90}
 \partial_{x_2}\mathcal G\zeta(\mathbf x)=\frac{1}{2\pi}\int_{\Pi}\frac{x_2+y_2}{|\mathbf x-\bar{\mathbf y}|^2}\zeta(\mathbf y)d\mathbf y+\frac{1}{2\pi}\int_\Pi \frac{y_2-x_2}{|\mathbf x-\mathbf y|^2}\zeta(\mathbf y)d\mathbf y.
 \end{equation}
Inserting \eqref{c90} into \eqref{crll} we obtain
\begin{equation}\label{h99}
\lambda_\zeta=\frac{1}{2\pi \kappa}\int_\Pi \int_{\Pi}\frac{x_2+y_2}{|\mathbf x-\bar{\mathbf y}|^2}\zeta(\mathbf x)\zeta(\mathbf y)d\mathbf xd\mathbf y+\frac{1}{2\pi\kappa}\int_\Pi\int_\Pi \frac{y_2-x_2}{|\mathbf x-\mathbf y|^2}\zeta(\mathbf x)\zeta(\mathbf y)d\mathbf xd\mathbf y.
\end{equation}
Observe that  the second integral on the right-hand side of \eqref{h99} is exactly zero due to the anti-symmetry of the integrand. Hence the proof is finished.
\end{proof}

Now we are ready to estimate the traveling speed $\lambda_\zeta$.
\begin{proposition}\label{prop20}
As $\varepsilon\to 0^+$ , $\lambda_\zeta\to  {\kappa^2}/(4\pi i_0)$ uniformly with respect to the choice of  $\zeta\in\Gamma_{\varepsilon,i_0}$.
\end{proposition}

\begin{proof}
Using Lemma \ref{cids} and Lemma \ref{ffs1}, we have
\begin{align*}
\lim_{\varepsilon\to0^+}\lambda_\zeta&=\lim_{\varepsilon\to0^+}\left(\frac{1}{2\pi\kappa}\int_\Pi\int_\Pi\frac{x_2+y_2}{|\mathbf x-\bar{\mathbf y}|^2}\zeta(\mathbf x)\zeta(\mathbf y)d\mathbf xd\mathbf y\right)\\
&=\frac{\kappa}{2\pi } \frac{x_2+y_2}{|\mathbf x-\bar{\mathbf y}|^2} \bigg|_{\mathbf x=\hat{\mathbf x},\,\mathbf y=\hat{\mathbf x}}\\
&=\frac{\kappa^2}{4\pi i_0}.
\end{align*}
Moreover, the convergence is uniform with respect to the choice of $\zeta\in\Gamma_{\varepsilon,i_0}$ since the convergence $\zeta\to\kappa\delta_{\hat{\mathbf x}}$ in Lemma \ref{cids} is uniform.
\end{proof}

As in Theorem \ref{thm1}, for $\zeta\in \Gamma^0_{\varepsilon,i_0}$, extend $\zeta$ to $\mathbb R^2$ such that $\zeta=0$ in the lower half-plane and define
\[\nu^{\zeta,\varepsilon}(\mathbf x)=\varepsilon^2\zeta(\varepsilon\mathbf x+\hat{\mathbf x}).\]

\begin{proposition}\label{prop60}
As  $\varepsilon\to0^+,$ $\nu^ {\zeta,\varepsilon}\to\varrho$ in $L^p(\mathbb R^2)$,  uniformly with respect to the choice of $\zeta\in \Gamma^0_{\varepsilon,i_0}$.

\end{proposition}
\begin{proof}
  It suffices to show that
for any sequence  $\{\varepsilon_n\}_{n=1}^{+\infty}\subset(0,\varepsilon_0)$ such that   $\varepsilon_n\to0^+$ as $n\to+\infty$, and any sequence  $\{\zeta_n\}_{n=1}^{+\infty}\subset  \Gamma_{\varepsilon_n,i_0}^0$, it holds that
\[\nu^{\zeta_n,\varepsilon_n}\to\varrho  \,\, \mbox{ in $L^p(\mathbb R^2)$ as }  n\to+\infty.\]
For simplicity, below we denote $\nu_n=\nu^{\zeta_n,\varepsilon_n}.$ It is easy to check that  for each $n$
\begin{equation}\label{nupp}
\nu_n\geq 0\mbox { a.e. in } \mathbb R^2,\quad\int_{\mathbb R^2}\mathbf x\nu_n(\mathbf x)d\mathbf x=\mathbf 0, \quad {\rm supp}(\nu_n)\subset  \overline{B_{C_{13}}(\mathbf 0)}.
\end{equation}
Besides, it is also clear that
\begin{equation}\label{lops}
\nu_n^*=\varrho\quad\forall\,n.
\end{equation}
Here $\nu_n^*$ denote the symmetric-decreasing rearrangement of $\nu_n.$
 Up to a subsequence, we assume that 
 \begin{equation}\label{gaila0}
\nu_n\rightharpoonup \nu \,\, \mbox{ in } L^p(\mathbb R^2) \mbox{ as } n\to+\infty.
\end{equation} 
Then it is clear that 
 \begin{equation}\label{ljj}
\int_{\mathbb R^2}\mathbf x\nu(\mathbf x)d\mathbf x=\mathbf 0, \quad {\rm supp}(\nu)\subset  \overline{B_{C_{13}}(\mathbf 0)}.
\end{equation}
Besides, observing that $\{\nu_n\}_{n=1}^{+\infty}$  is   bounded   in $L^2(D)$,  we deduce from \eqref{gaila0} that
 \begin{equation}\label{gaila1}
\nu_n\rightharpoonup \nu \,\, \mbox{ in } L^2(\mathbb R^2) \mbox{ as } n\to+\infty.
\end{equation}

Define
\begin{equation}\label{7ui}
v_n(\mathbf x)=\frac{1}{\varepsilon_n^2}\varrho\left(\frac{\mathbf x-\hat{\mathbf x}}{\varepsilon_n}\right).
\end{equation}
As in the proof of Lemma \ref{wll}, it is easy to check  that $v_n\in\mathcal S_{\varepsilon_n,i_0}$ for each $n$. Therefore 
$E(\zeta_n)\geq E(v_n),$
 that is,
 \begin{equation}
\frac{1}{4\pi}\int_\Pi\int_\Pi\ln\frac{|\mathbf x-\bar{\mathbf y}|}{|\mathbf x-\mathbf y|}\zeta_n(\mathbf x)\zeta_n(\mathbf y)d\mathbf xd\mathbf y \geq \frac{1}{4\pi}\int_\Pi\int_\Pi\ln\frac{|\mathbf x-\bar{\mathbf y}|}{|\mathbf x-\mathbf y|}v_n(\mathbf x)v_n(\mathbf y)d\mathbf xd\mathbf y,
 \end{equation}
or equivalently,
 \begin{equation}\label{sptt5}
 \begin{split}
 &\frac{1}{4\pi}\int_\Pi\int_\Pi\ln\frac{1}{|\mathbf x-\mathbf y|}\zeta_n(\mathbf x)\zeta_n(\mathbf y)d\mathbf xd\mathbf y+\frac{1}{4\pi}\int_\Pi\int_\Pi\ln {|\mathbf x-\bar{\mathbf y}|}\zeta_n(\mathbf x)\zeta_n(\mathbf y)d\mathbf xd\mathbf y\\
 \geq &\frac{1}{4\pi}\int_\Pi\int_\Pi\ln\frac{1}{|\mathbf x-\mathbf y|}v_n(\mathbf x)v_n(\mathbf y)d\mathbf xd\mathbf y+\frac{1}{4\pi}\int_\Pi\int_\Pi\ln|\mathbf x-\bar{\mathbf y}|v_n(\mathbf x)v_n(\mathbf y)d\mathbf xd\mathbf y.
 \end{split}
 \end{equation}
 By the definition of $v_n$, it is easy to check that as $n\to+\infty$
 \begin{equation}\label{3k1}
 \frac{1}{4\pi}\int_\Pi\int_\Pi\ln|\mathbf x-\bar{\mathbf y}|v_n(\mathbf x)v_n(\mathbf y)d\mathbf xd\mathbf y \to \frac{\kappa^2}{4\pi}\ln\left(\frac{2i_0}{\kappa}\right).
 \end{equation}
By  Lemma \ref{cids}, it is also clear that  as $n\to+\infty$
 \begin{equation}\label{3k2}
 \frac{1}{4\pi}\int_\Pi\int_\Pi\ln|\mathbf x-\bar{\mathbf y}|\zeta_n(\mathbf x)\zeta_n(\mathbf y)d\mathbf xd\mathbf y \to \frac{\kappa^2}{4\pi}\ln\left(\frac{2i_0}{\kappa}\right).
  \end{equation}
From  \eqref{3k1} and \eqref{3k2}, we can write \eqref{sptt5} as
  \begin{equation}\label{sptt6}
\int_\Pi\int_\Pi\ln\frac{1}{|\mathbf x-\mathbf y|}\zeta_n(\mathbf x)\zeta_n(\mathbf y)d\mathbf xd\mathbf y\geq \int_\Pi\int_\Pi\ln\frac{1}{|\mathbf x-\mathbf y|}v_n(\mathbf x)v_n(\mathbf y)d\mathbf xd\mathbf y+\gamma_n,
 \end{equation}
 where $\{\gamma_n\}_{n=1}^{+\infty}$ satisfies $\gamma_n\to 0 $ as $n\to+\infty.$
 Taking into account  the relation
 \[
 \zeta_n(\mathbf x)=\frac{1}{\varepsilon^2_n}\nu_n\left(\frac{\mathbf x-\hat{\mathbf x}}{\varepsilon_n}\right),
 \]
and the fact that $\nu_n\rightharpoonup \nu$ in $L^2(\mathbb R^2)$ as $n\to+\infty$, we can  pass  to the limit $n\to+\infty$  in \eqref{sptt6} to get
 \begin{equation}\label{k3k}
 \int_{\mathbb R^2}\int_{\mathbb R^2}\ln\frac{1}{|\mathbf x-\mathbf y|}\nu(\mathbf x)\nu(\mathbf y) d\mathbf xd\mathbf y\geq \int_{\mathbb R^2}\int_{\mathbb R^2}\ln\frac{1}{|\mathbf x-\mathbf y|}\varrho(\mathbf x)\varrho(\mathbf y) d\mathbf xd\mathbf y.
 \end{equation}
Now we can apply Lemma \ref{bgu}  by choosing $u_n=\nu_n$ and $u_n^*\equiv \varrho$ (this is possible by \eqref{lops}) to get $\nu=\varrho.$
 Note that the condition \eqref{lily} in Lemma \ref{bgu} is satisfied by \eqref{nupp}.

To conclude, we have obtained as $n\to+\infty$
\begin{equation}\label{m11}
\nu_n\rightharpoonup \varrho \mbox{  in $L^p(\mathbb R^2)$.}
\end{equation}
 On the other hand, in view of \eqref{lops} we have
\begin{equation}\label{m12}
\|\nu_n\|_{L^p(\mathbb R^p)}=\|\nu_n^*\|_{L^p(\mathbb R^p)}=\|\varrho\|_{L^p(\mathbb R^2)} \quad\forall\,n.
\end{equation}
From \eqref{m11} and \eqref{m12}, by uniform convexity we deduce  that
\begin{equation}\label{m13}
\nu_n\to \varrho\,\, \mbox{ in $L^p(\mathbb R^2)$ as $n\to+\infty.$ }
\end{equation}
Hence  the proposition has been proved.

\end{proof}

\begin{proof}[Proof of Theorem \ref{thm1}]
Items (1)-(5) have been proved in Proposition \ref{prop1}. Items (6)-(8) follow from Propositions \ref{lam203}, \ref{prop20} and \ref{prop60}, respectively.
\end{proof}

\section{Fine asymptotic estimates for Burton's vortex pairs}

As an application of Theorem \ref{thm1}, we are able to figure out the asymptotic behavior for Burton's vortex pairs in \cite{B10} as the impulse goes to infinity.

Recall that $\Sigma_\alpha$ is the set of maximizers of $E$ relative to $\mathcal S_{1,\alpha}.$   See Lemma \ref{lemm1}.
Denote $\Sigma^0_\alpha$ the  set of maximizers that are Steiner-symmetric in the $x_1$-axis, or equivalently, 
\begin{equation}\label{sigma0}
\Sigma^0_\alpha=\left\{\zeta\in\Sigma_\alpha\,\,\bigg|\,\, \int_\Pi x_1\zeta(\mathbf x)d\mathbf x=0\right\}.
\end{equation}
The main result of this section is  the following theorem.

\begin{theorem}\label{coro1}
Let $\alpha_0>0$ be determined in Lemma \ref{lemm1}. Then for any $\alpha>\alpha_0$, the following assertions  for $\Sigma_{\alpha}$ hold true.
\begin{itemize}
\item[(a)] There exists some  $C>0$, depending only on $\varrho$, such that
\[ {\rm diam}(V_\zeta)\leq C\quad\forall\,\zeta\in\Sigma_\alpha.\]
\item[(b)] Let $q_\zeta$ be determined by (iii) in Lemma \ref{lemm1}, then \[\alpha q_\zeta\to  \frac{\kappa^2}{4\pi}\quad\mbox{as }\alpha\to+\infty,\] uniformly for $\zeta\in\Sigma_\alpha^0$. More precisely, for any $\epsilon>0,$ there exists some  $\alpha_1>\alpha_0,$ depending only on $\varrho$ and $\epsilon,$  such that for any $\alpha>\alpha_1,$ it holds that
\[\left|\alpha q_\zeta- \frac{\kappa^2}{4\pi }\right|<\epsilon\quad\forall\,\zeta\in\Sigma^0_{\alpha}.\]

\item[(c)] For $\zeta\in\Sigma^0_\alpha$, extend $\zeta$ to $\mathbb R^2$ such that $\zeta=0$ in the lower half-plane and define  
\[\nu^{\zeta,\alpha}(\mathbf x)=\zeta(\mathbf x+\tilde{\mathbf x}_\alpha),\quad\tilde{\mathbf x}_\alpha=\left(0,\frac{\alpha}{\kappa}\right).\] Then $\nu^{\zeta,\alpha}\to\varrho$ in $L^p(\mathbb R)$ as $\alpha\to+\infty,$ uniformly for $\zeta\in\Sigma_\alpha^0.$ More precisely, for any $\epsilon>0,$ there exists some  $\alpha_2>\alpha_0,$ such that for any $\alpha>\alpha_2,$ it holds that
\[\|\nu^{\zeta,\alpha}-\varrho\|_{L^p(\mathbb R^2)}<\epsilon\quad\forall\,\zeta\in\Sigma^0_\alpha.\]

\end{itemize}
\end{theorem}

\begin{proof}
Take $\alpha=i_0/\varepsilon.$ By Lemma \ref{lam1}, it holds that
\[\Sigma_\alpha=\Gamma_{1,i_0/\varepsilon}=\{v\in L^\infty(\Pi)\mid v^{\varepsilon}\in \Gamma_{\varepsilon, i_0}\}.\]
Therefore (6)-(8) in Theorem \ref{thm1}  can be equivalently expressed as follows:
\begin{itemize}
\item[(6)'] ${\rm diam}(V_ {\zeta^\varepsilon})\leq C\varepsilon$ for any $\zeta \in\Sigma_\alpha$, where $C>0$ does not depend on $\alpha$ (or $\varepsilon$);
\item[(7)'] $\lambda_{\zeta^\varepsilon}\to {\kappa^2}/{(4\pi i_0)}$ uniformly for any $\zeta\in\Sigma_\alpha,$ where $\lambda_{\zeta^\varepsilon}$ is determined by item (3) in Theorem \ref{thm1};
\item[(8)'] as $\varepsilon\to0^+$, $\varepsilon^2\zeta^\varepsilon(\varepsilon\cdot+\hat{\mathbf x})\to \varrho$ in $L^p(\mathbb R^2)$,  uniformly for any $\zeta\in\Sigma_\alpha.$
\end{itemize}
To prove (a), it suffices to use (6)' and the following fact
\[V_{\zeta^\varepsilon}=\varepsilon V_\zeta:=\{\varepsilon\mathbf x\mid \mathbf x\in V_\zeta\}.\]
 To prove (b), observe that
\begin{equation}\label{aiyo}
\zeta^\varepsilon =f_{\zeta^\varepsilon}(\mathcal G\zeta^\varepsilon -\lambda_{\zeta^\varepsilon}x_2) \mbox{ a.e. }\mathbf x\in \Pi \Longleftrightarrow  \zeta =\varepsilon^2 f_{\zeta^\varepsilon}(\mathcal G\zeta -\varepsilon\lambda_{\zeta^\varepsilon}x_2) \,\,\mbox{ a.e. }\mathbf x\in \Pi.
\end{equation}
Here we used the fact that
\[\mathcal G\zeta^\varepsilon(\mathbf x)=\mathcal G\zeta\left(\frac{\mathbf x}{\varepsilon}\right) \quad\forall\,\mathbf x\in\Pi.\]
On the other hand, by item (iii) in Lemma \ref{lam1}, it holds that
\begin{equation}\label{aiya}
\zeta=\phi_\zeta(\mathcal G\zeta-q_\zeta x_2) \,\,\mbox{  a.e. }\mathbf x\in \Pi.
\end{equation}
Comparing \eqref{aiyo} and \eqref{aiya} we get
\[\phi_\zeta=\varepsilon^2 f_{\zeta^\varepsilon},\quad q_\zeta=\varepsilon\lambda_{\zeta^\varepsilon},\]
which together with (7)'  yields  (b). For (c),   observe that
\[\varepsilon^2\zeta^\varepsilon(\varepsilon\mathbf x+\hat{\mathbf x})=\zeta(\mathbf x+\varepsilon^{-1}\hat{\mathbf x})=\zeta (\mathbf x+ \tilde{\mathbf x}_\alpha)\quad \forall\,\mathbf x\in\mathbb R^2.\]
Therefore  (c) follows from (8)' immediately.
\end{proof}

\section{Stability of Chaplygin-Lamb dipole}

First we recall the explicit expression of the Chaplygin-Lamb dipole.
Define
\begin{equation}
\psi_c(\mathbf x)=
\begin{cases}
x_2-\frac{2J_1(|\mathbf x|)}{J_1'(a)|\mathbf x|}x_2,&\mathbf x\in\Pi,\,\, |\mathbf x|\leq a,\\
\frac{a}{|\mathbf x|^2}x_2,&\mathbf x\in\Pi, \,\,|\mathbf x|\geq a,
\end{cases}
\end{equation}
where $J_k, k=0,1,$ is the $k$-th order Bessel function of the first kind, and  $a$ is the first positive zero of $J_1$. Then $\psi_c\in W^{2,s}_{\rm loc}(\Pi)$ and satisfies 
\begin{equation}\label{hsbb}
-\Delta\psi=(\psi-x_2)^{+}\quad \mbox{ a.e. in }\Pi.
\end{equation}
Define 
\[\zeta_c=-\Delta \psi_c.\]
In view of \eqref{hsbb}, it is easy to check that $\zeta_c(x_1-t, x_2)$ solves the vorticity \eqref{vor1}, thus corresponds to a traveling vortex pair with unit speed, generally referred to as  the Chaplygin-Lamb dipole.
Note that
\begin{equation}\label{lopao}
 I(\zeta_c)=\pi a^2,\quad \|\zeta_c\|^2_{L^2(\Pi)}=\pi a^2.
\end{equation}
See \cite{B50}, Lemma 6 for example.

Our purpose in this section is to prove the following  theorem concerning the stability of the  Chaplygin-Lamb dipole.
\begin{theorem}\label{cllb}
Let $2<s<+\infty$ be fixed. Define 
\begin{equation}
\mathcal C=\left\{ \zeta_c( x_1-\beta, x_2) \mid \beta\in\mathbb R\right\}.
\end{equation}
Then $\mathcal C$ is stable in the following sense: for any $\epsilon>0,$ there exists some $\delta>0$, such that
 for any  $L^s$-regular solution $\omega$ to the vorticity equation \eqref{vor1} satisfying  
 \[\omega(0,\cdot) \mbox{ \rm is nonnegative with  bounded support,}\quad \inf_{\zeta\in \mathcal C}\|\omega(0,\cdot)-\zeta\|_{\mathfrak X_s}<\delta,\]
 it holds that
\[\inf_{\zeta\in \mathcal C}\|\omega(t,\cdot)-\zeta\|_{\mathfrak X_s}<\epsilon\quad\forall \,t\ge 0.\]
Here the norm  $\|\cdot\|_{\mathfrak X_s}$ is defined as in Theorem \ref{thm1}.
\end{theorem}

\begin{remark}
Abe-Choi \cite{AC} recently proved a similar stability result  based on a new variational characterization of the Chaplygin-Lamb dipole. However,  the initial perturbation class and the norm of measuring stability adopted by  Abe-Choi  are different from the ones in Theorem \ref{cllb}.
\end{remark}

The proof of Theorem \ref{cllb} is based on a stability theorem proved by Burton in \cite{B10} and a variational characterization proved by Burton in \cite{B50}. They are stated as follows.

\begin{theorem}[\cite{B10}, Theorem 2]\label{bffk}
Let $2<s<+\infty,$ $v_0\in L^p(\Pi)$ be nonnegative with bounded support, and $\gamma>0$ be fixed.
Let  $\mathcal R(v_0)$  be the rearrangement class of $v_0$ in $\Pi$, and 
 $\overline{\mathcal R^W(v_0)}$  be the closure of  $\mathcal R(v_0)$ in the weak topology of $L^s(\Pi)$. Denote by $\Sigma$ the set of maximizers of $E$ subject to
\[\left\{v\in\overline{\mathcal R^W(v_0)}\mid I(v)=\gamma\right\}.\]
If $\varnothing\neq \Sigma\subset\mathcal R(v_0),$ then $\Sigma$ is stable as in  Theorem \ref{cllb}.
\end{theorem}

\begin{theorem}[\cite{B50}, p. 76]\label{bfkk}
The functional $E-I$ attains it maximum value subject to
\[\left\{v\in L^2(\Pi)\mid v\geq 0 \,\,\mbox{\rm a.e., }\,\, I(v)<+\infty,\,\,\|v\|_{L^2(\Pi)}\leq \pi^{1/2} a\right\},\]
and any maximizer must be of the form  
\[\alpha\zeta_c(\alpha(x_1-\beta),\alpha x_2)\]
 for some  $\alpha\geq 0$ and some $\beta\in\mathbb R$.
\end{theorem}
Note that  Theorem \ref{bfkk} did not appear as a theorem in \cite{B50}, but  can be found  at the end of Section 3, p. 76.

Now we are ready to prove Theorem \ref{cllb}.
\begin{proof}[Proof of Theorem \ref{cllb}]
By Theorem \ref{bffk} it suffices to show that $\mathcal C$ is exactly the set of maximizers of $E$ subject to 
\[\mathcal A:=\left\{v\in\overline{\mathcal R^W(\zeta_c)}\mid I(v)=\pi a^2\right\}.\]
To this end, define
\[\mathcal B:=\left\{v\in L^2(\Pi)\mid v\geq 0 \,\,\mbox{\rm a.e., }\,\, I(v)=\pi a^2,\,\,\|v\|_{L^2(\Pi)}\leq \pi^{1/2} a\right\},\]
In view of \eqref{lopao} and Theorem \ref{bfkk}, it is easy to see that $\mathcal C$ is  exactly the set of maximizers of $E-I$ over $\mathcal B.$ Therefore $\mathcal C$ is  exactly the set of maximizers of $E$ over $\mathcal B.$ Taking into account the fact that $\mathcal C\subset \mathcal A\subset \mathcal B,$
 we immediately deduce that $\mathcal C$ is  exactly the set of maximizers of $E$ over $\mathcal A.$ Hence the proof is finished.

\end{proof}

\begin{remark}
By the proof of Theorem \ref{cllb}, $\mathcal C$ is in fact the set of maximizers of $E$  subject to
\[ \left\{v\in \mathcal R(\zeta_c)\mid I(v)=\pi a^2\right\}.\]
This means that for some special rearrangement class and impulse, existence and stability in Theorem \ref{thm1} also holds true, even if $\varepsilon$ is not small. However, for general  rearrangement class and impulse the theory is far from completed.
\end{remark}

{\bf Acknowledgements:}
{G. Wang was supported by the National Natural Science Foundation of China grants (12001135, 12071098), and the China Postdoctoral Science Foundation grants (2019M661261, 2021T140163). }

\phantom{s}
 \thispagestyle{empty}

\end{document}